\documentclass[10pt]{amsart}

\usepackage{amsmath}
\usepackage{amsthm} %breqn loads amsmath
\usepackage{graphicx}
\usepackage{color}
\usepackage{scalerel}

\usepackage{enumerate}
\usepackage{amsfonts}
\usepackage{amssymb}
\usepackage[hidelinks]{hyperref}
\usepackage[capitalize]{cleveref}

\usepackage{nicefrac}
\usepackage{xfrac}

\usepackage{tipa}
\usepackage[normalem]{ulem}

\usepackage{mathtools}

\DeclareMathOperator*{\forkindep}{\raise0.2ex\hbox{\ooalign{\hidewidth$\vert$\hidewidth\cr\raise-0.9ex\hbox{$\smile$}}}}

\newcommand{\Aut}{\operatorname{Aut}}

\newcommand{\tp}{\operatorname{tp}}

\newcommand{\conv}{\operatorname{conv}}

\DeclareMathAlphabet\mathbfcal{OMS}{cmsy}{b}{n}

\newtheorem*{claim-star}{Claim}
\newtheorem{theorem}{Theorem}[section] % numbered like the section
\newtheorem{lemma}[theorem]{Lemma}

\newtheorem{prop-def}[theorem]{Proposition-Definition}
\newtheorem{corollary}[theorem]{Corollary}
\newtheorem{fact}[theorem]{Fact}
\newtheorem{fact-eh}[theorem]{Fact(?)}

\newtheorem{question}[theorem]{Question}
\newtheorem{proposition}[theorem]{Proposition}
\newtheorem{proposition-eh}[theorem]{Proposition(?)}
\newtheorem*{theorem-star}{Theorem}
\newtheorem*{conjecture-star}{Conjecture}
\newtheorem*{lemma-star}{Lemma}

\theoremstyle{definition}
\newtheorem{definition}[theorem]{Definition}
\newtheorem{example}[theorem]{Example}

\newtheorem{remark}[theorem]{Remark}
\theoremstyle{remark}

\newtheorem*{warning}{Warning}

\newcommand{\fs}{\mathrm{fs}}

\newcommand{\inv}{\mathrm{inv}}
\newcommand{\supp}{\mathrm{supp}}

\newcommand{\Av}{\mathrm{Av}}

\allowdisplaybreaks %Fix horrible vertical spacing issues.

\title{A note on Transfer maps and the Morley product in NIP theories}

\author{ Kyle Gannon}
\address{Beijing International Center for Mathematical Research (BICMR)\\
Peking University\\
Beijing, China.}
\email{kgannon@bicmr.pku.edu.cn}
\begin{document}
\maketitle
\begin{abstract}
    In an important (yet unpublished) research note, Ben Yaacov describes how to turn a global Keisler measures into a type over a monster model of the randomization. This transfer methods allow one to turn questions involving measures into those involving types (in continuous logic). Assuming that $T$ is NIP, we show that the Morley product commutes with the transfer map for finitely satisfiable measures. We characterize when the Morley product commutes with the restriction map for pairs of global finitely satisfiable types in the randomization. We end by making some brief observations about the Ellis semigroup in this context. 
\end{abstract}

\section{Introduction \& Preliminaries}

\subsection{Introduction} While many of the properties of types have generalizations to the context of measures (e.g., invariance, definability, finite satisfiability), the \emph{machine that is model theory} interacts with these objects differently. The central difference is that \emph{you can't realize a measure}, in a strong sense. Measures are not realized in some elementary extension the way types can be realized. Granted, a lot can be accomplished in this setting, but some powerful classical tools are missing. One striking difference is the lack of arguments using indiscernible sequences. This is because, in a literal sense, there are none.

The power of the randomization \cite{keisler1999randomizing,ben2009randomizations} is that it provides a framework in which one can \emph{realize a measure}, if you are willing to change your underlying structure to a closely related one, and also learn a smidgen of continuous logic. Instead of working with measures directly, one now works with \emph{random variables} (or \emph{random elements}) taking values in our original structure. This perspective restores access to a lot of model theoretic machinery, including important portions of neostability \cite{ben2010continuous,yaacov2008continuous,ben2009randomizations,andrews2019independence,CGH2,khanaki2022generic}. More importantly, this perspective is practical and we can prove new results about measures (this is often the case with particular kinds of convergence problems -- See \cite[Lemma 2.10]{NIP3} via \cite{yaacov2008continuous}, or our forth coming preprint with Chernikov and Krupi\'{n}ski). 

The Morley product is a fundamental construction in model theory, for both types and measures \cite{NIP3}. When it exists, this product glues measures together in a generic/non-forking way. This product plays an important role in combinatorial results \cite{NIP4,simon2016note}, convolution algebras and dynamical systems (\cite{chernikov2022definable,chernikov2023definable}), as well as certain kinds of probability spaces (forth coming with Hanson). Understanding precisely how the Morley product interacts with the transfer maps gives us access to more tools to better analyze these constructions and contexts. 

The results here are primarily about the connections between measures and random types over NIP structures. We will assume that our underlying theory $T$ is both countable and NIP throughout. For some technical reasons, we focus here on the case of finitely satisfiable measures instead of the more general setting of invariant measures. Even though we make this general assumption, we try our best to specify precisely where we use the NIP and finite satisfiability assumptions. We do the following: 
\begin{enumerate}
    \item \emph{Basic Observations:} We verify some of the claims in Ben Yaacov's note  \cite{yaacov2009transfer}. We also make some basic observations. 
    \item \emph{Going up:} Assuming NIP, we prove that the Morley product commutes with transfer map.
    \item \emph{Going down:} We observe that the restriction map does not commute with the Morley product. We provide a condition on pairs of finitely satisfiable types which ensures commutativity.
    \item \emph{Over a group:} The transfer map is an injective embedding of topological semigroups from the convolution algebra of Keisler measures to types over the randomization (under the \emph{Newelski product}). If our group is definably amenable, then the minimal left ideal in the space of finitely satisfiable types over the randomization is a single point. We had hoped that the transfer map would provide an isomorphism on the level of minimal left ideals in the general NIP setting. However, we show that if our underlying group is not definably amenable, then the minimal left ideal in the Ellis semigroup of random finitely satisfiable types is somehow always more complicated than the minimal left ideal in the convolution algebra.
\end{enumerate}

\subsection{Preliminaries}We assume some familiarity with Keisler measures, randomizations, \emph{basic} continuous logic, etc. Fix a first order theory $T$ in a countable language $\mathcal{L}$. We briefly recall the construction of \emph{easy models} of the randomization. Let  $(\Omega_0, \mathcal{B}_0,\mathbb{P}_0)$ denote an atomless probability algebra with event space $\Omega_0$, $\sigma$-algebra $\mathcal{B}_0$, and probability measure $\mathbb{P}_0$. Then any model $M\models T$ gives rise to a model $M^{(\Omega_0,\mathcal{B}_0,\mathbb{P}_0)}$ of the randomization of $T$, denoted $T^{R}$. We will write $M^{(\Omega_0,\mathcal{B}_0,\mathbb{P}_0)}$ simply as $M^{\Omega}$. We recall its construction. 
First define 
\[ 
M_0^\Omega := \{f\colon \Omega \to M: f \text{ is } \mathcal{B}_0\text{-measurable and has finite image}\}.
\] 
Notice that $M^{\Omega}_0$ has a natural pseudo-metric space given by $d(f,g) = \mathbb{P}_0(\{ t \in \Omega_0 : f(t) \neq g(t)\})$. Then $M^{\Omega}$ is constructed by quotienting by distance $0$ and taking the metric completion on the quotient. Likewise, $\mathcal{B}_0$ is a pseudo-metric space, with distance given by $d(A,B) = \mathbb{P}_0(A \triangle B)$. Again, we identifying measurable sets of up measure $0$ and take the metric completion. Then the randomization $(M^\Omega,\mathcal{B})$ is the two sorted structure in the language $\mathcal{L}^{R}$ which is the disjoint union of the language of probability algebras on one sort and a collection of functions, $\{[\varphi(\bar{x})]: \varphi(\bar{x})$ is an $\mathcal{L}$-formula$\}$, from the random element sort to the probability algebra sort. More explicitly, 
\begin{enumerate}
    \item $\mathcal{B}$ retains the structure of a probability algebra in the language $\{\mathbb{P}, \cap, \cup, ^{c}, \perp, \top\}$. In particular, formulas are extended from $\mathcal{B}_0$ to $\mathcal{B}$ via continuity.   
    \item For every $\mathcal{L}$-formula $\varphi(x_1,...,x_n)$, we have a function $[\varphi(x_1,...,x_n)]:(M_0^{\Omega})^{n} \to \mathcal{B}_0$ via 
    \begin{equation*}
        [\varphi(h_1,...,h_n)] = \{t \in \Omega_0 : M \models \varphi(h_1(t),...,h_n(t))\}. 
    \end{equation*}
    This is extended by continuity to $(M^{\Omega},\mathcal{B})$. 
    \item Putting $(1)$ and $(2)$ together; For any $\mathcal{L}$-formula $\varphi(x_1,...,x_n)$ we have a function $\mathbb{E}[\varphi(x_1,...,x_n)]:(M_0^{\Omega})^{n} \to [0,1]$ via
    \begin{equation*}
    \mathbb{E}[\varphi(h_1,...,h_n)] = \mathbb{P}_0(\{t \in \Omega_0: M \models \varphi(h_1(t),...,h_n(t))\}). 
\end{equation*}
    Again, this is extended by continuity to $(M^{\Omega},\mathcal{B})$. We often only need to work with $\mathcal{L}^{R}$-formulas of the form $\mathbb{E}[\varphi(x_1,...,x_n)]$. 
\end{enumerate}

Formally, one can define the theory of the randomization of $T$, $T^{R}$, as the theory of $(M^{\Omega},\mathcal{B})$ in the language $\mathcal{L}^{R}$ for any choice of model of $T$ and any atomless probability algebra $(\Omega_0,\mathcal{B}_0,\mathbb{P}_0)$ \cite{ben2009randomizations}. By construction, the structure $(M_0^\Omega,\mathcal{B}_0)$ is a metrically dense (pre-)substructure of $(M^{\Omega},\mathcal{B})$. Throughout this paper, we will usually refer to an \emph{easy model} by its \emph{random element sort}, e.g. $\mathbfcal{M}^{\Omega} := (M^{\Omega},\mathcal{B})$, and leave the \emph{probability algebra} sort implicit.

We remark that if $\mathcal{U}$ is a monster model of $T$, then the model $\mathbfcal{U}^{\Omega} = (\mathcal{U}^{\Omega},\mathcal{B})$ is usually not saturated and so we will always think of $\mathbfcal{U}^{\Omega}$ as elementarily embedded in a monster model of $T^{R}$, i.e., $\mathbfcal{U}^{\Omega} \prec \mathbfcal{C} = (\hat{\mathcal{K}}_{\mathbfcal{C}},\hat{\mathcal{B}}_{\mathbfcal{C}})$. As convention, we write $b \in \mathbfcal{C}$ to mean $b \in \hat{\mathcal{K}}_{C}$ as well as $p \in S_{x}(\mathbfcal{C})$ to mean that $p$ is a global type in the \emph{random element} sort, i.e. $p \in S_{\mathcal{K}}(\mathbfcal{C})$. 

If $h \in \mathcal{U}^{\Omega}_0$, then a \emph{partition for h} is a finite $\mathcal{B}_0$-measurable partition of $\Omega_0$, say  $\mathcal{A}$, such that such that for any $A \in \mathcal{A}$, $h|_{A}: A \to \mathcal{U}$ is a constant function. We let $h|_{A}$ denote the unique element in $\mathcal{U}$ such that for any $t \in A$, $h(t) = h|_{A}$. By elementarity, one can interpret the set $\mathcal{A}$ as also a partition of the probability algebra associated to  $\mathbfcal{C}$, i.e. $\mathcal{B}_{C}$. If $b \in \mathcal{U}$, we let $f_{b}: \Omega \to \mathcal{U}$ define the random variable which takes constant value $b$, i.e. for any $t \in \Omega$, $f_b(t) = b$.

Our set up is as follows: Let $M$ be a model of $T$, $\mathcal{U}$ be a monster model of $T$ such that $M \prec \mathcal{U}$. Then we have that $\mathbfcal{M}^{\Omega} \prec \mathbfcal{U}^{\Omega}$ (by quantifier elimination, see Remark \ref{remark:tools}). We now consider a monster model $\mathbfcal{C}$ of $T^{R}$ such that $\mathbfcal{M}^{\Omega} \prec \mathbfcal{U}^{\Omega} \prec \mathbfcal{C}$. We are interested in the connection between global types in the \emph{random element sort} which are finitely satisfiable over the structure $\mathbfcal{M}^{\Omega}$, denoted  $S_{x}^{\fs}(\mathbfcal{C},\mathbfcal{M}^{\Omega})$ and global Keisler measures which are finitely satisfiable over $M$, denoted $\mathfrak{M}_{x}^{\fs}(\mathcal{U},M)$. We let $\mathfrak{M}_{x}(\mathcal{U})$ be the collection of all global Keisler measures over $\mathcal{U}$. 

We recall the maps between these two spaces. Our first map is the restriction map from $S_{x}(\mathbfcal{C})$ to $\mathfrak{M}_{x}(\mathcal{U})$.

\begin{definition} We define the map $\nu_{-}: S_{x}(\mathbfcal{C}) \to \mathfrak{M}_{x}(\mathcal{U})$ where for any $q \in S_{x}(\mathbfcal{C})$, 
\begin{equation*}
\nu_{q}(\varphi(x,b_1,...,b_n)) = (\mathbb{E}[\varphi(x,f_{b_1},...,f_{b_n})])^{q}. 
\end{equation*} 
\end{definition} 

\begin{proposition}\label{prop:cont-v} The map $\nu_{-}: S_{x}(\mathbfcal{C}) \to \mathfrak{M}_{x}(\mathcal{U})$ is continuous. 
\end{proposition}

\begin{proof}
Notice that if $O$ is a basic open set in $\mathfrak{M}_{x}(\mathcal{U})$ is of the form
\begin{equation*}
    O=\bigcap_{i=1}^{n} \{\mu \in \mathfrak{M}_{x}(\mathcal{U}):r_{i}<\mu(\varphi_{i}(x,b_{i}))<s_{i}\},
\end{equation*}
where $r_1,...,r_n,s_1,...,s_n$ are real numbers, $b_1,...,b_n$ are parameters from $\mathcal{U}$, and $\varphi_1(x,y),...,\varphi_n(x,y)$ are $\mathcal{L}$-formulas. Then 
\begin{equation*}
    \left(\nu_{-}\right)^{-1}(O) = \bigcap_{i=1}^{n}\{p\in S_{x}(\mathbfcal{C}): r_i<(\mathbb{E}[\varphi_i(x,f_{b_i})])^{p}<s_{i}\},
\end{equation*}
which is clearly open. 
\end{proof}

The following two constructions are essentially from \cite{yaacov2009transfer}, the only caveat being that the map $r_{-}$ has codomain over the monster model. We also mention that the map $r_{-}$ was explicitly studied in \cite{CGH2}. Ben Yaacov calls the following map the \emph{natural extension}. 

\begin{definition} Let $\mu \in \mathfrak{M}_{x}(\mathcal{U})$ and suppose that $\mu$ is definable over $M$. Then there exists a unique type $r_{\mu} \in S^{\inv}_{x}(\mathbfcal{C},\mathbfcal{M}^{\Omega})$ such that for every $b \in \mathbfcal{C}^{y}$, and $\varphi(x,y) \in \mathcal{L}_{xy}$, 
\begin{equation*}
(\mathbb{E}[\varphi(x,b)])^{r_{\mu}} = \lim_{i \in I} \sum_{A \in \mathcal{A}_i} \mathbb{P}_0(A) \mu(\varphi(x,h_i|_{A})).
\end{equation*} 
where $(h_i,\mathcal{A})_{i \in I}$ is any indexed family such that
\begin{enumerate} 
\item For each $i \in I$, $h_i \in (M_{0}^{\Omega})^{y}$. 
\item For each $i \in I$, $\mathcal{A}_i$ is a partition for $h_i$. 
\item $\lim_{i \in I} \tp(h_i/\mathbfcal{M}^{\Omega}) = \tp(b/\mathbfcal{M}^{\Omega})$. 
\end{enumerate} 
\end{definition} 

The following map was also defined in \cite{yaacov2009transfer}. Ben Yaacov refers to it as \emph{extension-by-definition}. This paper will primarily focus on this map. We remark that \cite{yaacov2009transfer} defines the map for all Borel-definable measures. However, we only claim to verify that this map is well-defined under an NIP+finitely satisfiable assumption and so we have changed the definition to fit our context. 

\begin{definition}[T NIP] Let $\mu \in \mathfrak{M}_{x}^{\fs}(\mathcal{U},M)$. Then there exists a unique type $s_{\mu} \in S^{\fs}_{x}(\mathbfcal{C},\mathbfcal{M}^{\Omega})$ such that for every $b \in \mathbfcal{C}^{y}$, and $\varphi(x,z) \in \mathcal{L}_{xz}$, 
\begin{equation*}
(\mathbb{E}[\varphi(x,b)])^{s_{\mu}} = \int_{S_{y}(M)} F_{\mu}^{\varphi} d \nu_{\tp(b/\mathbfcal{C})}. 
\end{equation*} 
where $F_{\mu}^\varphi: S_{y}(M) \to [0,1]$ via $F_{\mu}^{\varphi}(q) = \mu(\varphi(x,e))$ where $e \models q$. We remark that $F_{\mu}^{\varphi}$ is Borel by Remark \ref{remark:tools}. 
\end{definition} 

As stated previously, we assume that the reader has some knowledge of Keisler measures and randomizations, especially when it comes to some of the technical definitions involving properties of measures. Most of these can be found in Simon's text \cite{Guide}. We recall the definition of the Morley product: 

\begin{definition} Suppose that $\mu \in \mathfrak{M}_{x}(\mathcal{U})$ and $\mu$ is Borel-definable over $M$. Let $\nu \in \mathfrak{M}_{y}(\mathcal{U})$. Then the Morley product, denote $\mu \otimes \nu$, is the unique measure in $\mathfrak{M}_{xy}(\mathcal{U})$ such that for any $\mathcal{L}$-formula $\varphi(x,y,z)$ and parameter $e \in \mathcal{U}^{z}$, 
\begin{equation*}
    (\mu \otimes \nu)(\varphi(x,y,e)) = \int_{S_{x}(Me)} F_{\mu}^{\varphi} d\nu. 
\end{equation*}
\end{definition}
The following tools will aid us in our computations: 

\begin{remark}\label{remark:tools}
    \begin{enumerate}
    \item \emph{Equivalence of invariance and Borel-definability in NIP theories} \cite[Corollary 4.9]{NIP2}. Suppose that $T$ is NIP and $\mu \in \mathfrak{M}_{x}(\mathcal{U})$. Then $\mu$ is $M$-invariant if and only if $\mu$ is Borel-definable over $M$, i.e. for any $\mathcal{L}$-formula $\varphi(x,y)$, the map 
    \begin{equation*}
        F_{\mu}^{\varphi}:S_{y}(M) \to [0,1] \text{ via } F_{\mu}^{\varphi}(p) = \mu(\varphi(x,b)),
    \end{equation*}
    where $b \models p$ is well-defined and Borel. 
    \item \emph{Associativity of the Morley product of definable measures in arbitrary theories} \cite[Proposition 2.6]{CG}.
    Suppose that $T$ is arbitrary and $\mu \in \mathfrak{M}_{x}(\mathcal{U})$, $\nu \in \mathfrak{M}_{y}(\mathcal{U})$, and $\lambda \in \mathfrak{M}_{x}(\mathcal{U})$. If $\mu$ and $\nu$ are definable over a small model then $(\mu \otimes (\nu \otimes \lambda)) = ((\mu \otimes \nu) \otimes \lambda)$. 
    \item \emph{Associativity of invariant measures in NIP theories} \cite[Theorem 2.2]{GanCon2}. Suppose that $T$ is NIP and $\mu \in \mathfrak{M}_{x}(\mathcal{U})$, $\nu \in \mathfrak{M}_{y}(\mathcal{U})$, and $\lambda \in \mathfrak{M}_{x}(\mathcal{U})$. If $\mu$ and $\nu$ are invariant over a small model then $(\mu \otimes (\nu \otimes \lambda)) = ((\mu \otimes \nu) \otimes \lambda)$. 
    \item \emph{Left-continuity of the Morley product for invariant measures in NIP theories} (i.e.,\cite[Theorem 6.3]{chernikov2022definable}). Suppose that $T$ is NIP, $\nu \in \mathfrak{M}_{x}(\mathcal{U})$, and $\varphi(x,y) \in \mathcal{L}_{xy}(\mathcal{U})$. Then $- \otimes \nu (\varphi(x,y)) : \mathfrak{M}_{x}^{\fs}(\mathcal{U},M) \to [0,1]$ is continuous. As consequence, the map $-\otimes \nu : \mathfrak{M}_{x}^{\fs}(\mathcal{U},M) \to \mathfrak{M}_{xy}(\mathcal{U})$ is continuous. 
    \item \emph{Right-continuity of the Morley product for definable measures in arbitrary theories} (See e.g., \cite[Lemma 5.4]{CGH}). Suppose that $T$ is arbitrary, $\mu \in \mathfrak{M}_{x}(\mathcal{U})$, is definable $\varphi(x,y) \in \mathcal{L}_{xy}(\mathcal{U})$. Then $\mu \otimes - (\varphi(x,y)) : \mathfrak{M}_{y}(\mathcal{U}) \to [0,1]$ is continuous. 
    \item \emph{Quantifier Elimination in the Randomization} (i.e., \cite{ben2009randomizations}). $T^{R}$ has quantifier elimination. In particular, any type $p \in S_{x}(\mathbfcal{C})$ is completely determined by the values of 
    \begin{equation*}
        \{(\mathbb{E}[\varphi(x,h)])^{p}: \varphi(x,y) \in \mathcal{L}_{xy}(\emptyset), h \in \mathbfcal{C}^{y}\}. 
    \end{equation*} 
\end{enumerate}
\end{remark}

\subsection*{Acknowledgement} We thank Ita\"{i} Ben Yaacov for discussion. We are also deeply indebted to James Hanson for shaping our intuition about the subject.

\section{Basic Observations}

For the rest of the note, we fix $M$, $\mathcal{U}$ models of $T$ such that $M \prec \mathcal{U}$ and $\mathcal{U}$ is a monster model of $T$. We let $(\Omega_0,\mathcal{B}_0,\mathbb{P}_0)$ be a fixed atomless probability space and so $\mathbfcal{M}^{\Omega} \prec \mathbfcal{U}^{\Omega}$ is fixed. Finally, we fix a monster model $\mathbfcal{C}$ of the randomization $\mathbfcal{C} = (\hat{\mathcal{K}}_{\mathbfcal{C}},\hat{\mathcal{B}}_{\mathbfcal{C}})$.
We list the known results connected to transfer properties. The following statements are true. Many appear in either \cite{yaacov2009transfer} or \cite{CGH2}. 
\begin{enumerate}[(a)]
    \item Suppose that $\mu \in \mathfrak{M}_{x}(\mathcal{U})$ and $\mu$ is definable over $M$. Then $r_{\mu} = s_{\mu}$ (observed in \cite{yaacov2009transfer}, without proof). We take the opportunity to provide one. This is true when both $r_{\mu}$ and $s_{\mu}$ are well-defined and does not require NIP. 
    \item The map $\nu_{-}$ preserves invariance, definability, finite satisfiability, and fam. These are quite straightforward and we write them down here. 
    \item The map $r_{-}$ preserves definability, fam, and fim. Definability was proved in \cite{yaacov2009transfer} while fim and fam were proved in \cite{CGH2}. 
    \item (T NIP) The map $s_{-}$ preserves finite satisfiability. This was proved in \cite{yaacov2009transfer} and we give another proof here. By Statements $(a)$ and $(c)$, $s_{-}$ preserves definability and generic stability as well. 
\end{enumerate}
In this section, we give proofs of the following, in the order listed:
\begin{enumerate}
    \item A proof of Statement (a). 
    \item Under NIP, if $\mu \in \mathfrak{M}_{x}^{\fs}(\mathcal{U},M)$, then $s_{\mu}$ is consistent and finitely satisfiable in $\mathbfcal{M}^{\Omega}$. The latter is done in \cite{yaacov2009transfer} and our proof is a different style.
    \item We provide a proof of Statement (b). 
\end{enumerate}

We first prove that the maps $r_{-}$ and $s_{-}$ agree on definable measures. We recall a basic fact. 

\begin{fact}\label{Fact:easy-1} If $h \in \mathcal{U}^{\Omega}_0$ and $\mathcal{A}$ is a partition for $h$, then $\nu_{\tp(h/\mathbfcal{C})} = \sum_{A \in \mathcal{A}} \mathbb{P}_0(A) \delta_{h|_{A}}$.
\end{fact} 

\begin{proof} Fix $\varphi(x,y) \in \mathcal{L}_{xy}$ and $b \in \mathcal{U}^{y}$. Notice, 
\begin{align*} 
\nu_{\tp(h/\mathbfcal{C})}(\varphi(x,b)) &= (\mathbb{E}[\varphi(x,f_{b})])^{\tp(h/\mathbfcal{C})}\\
&= (\mathbb{E}[\varphi(h,f_{b})])\\
&= \mathbb{P}_0\Big(\Big\{ t \in \Omega_0: \mathcal{U} \models \varphi(h(t),f_{b}(t)) \Big\} \Big) \\
&= \sum_{A \in \mathcal{A}} \mathbb{P}_0\Big(\Big\{ t \in A: \mathcal{U} \models \varphi(h|_{A},b)) \Big\} \Big) \\
&= \sum_{A \in \mathcal{A}} \mathbb{P}_0(A)\delta_{h|_{A}}(\varphi(x,b)). \qedhere
\end{align*} 
\end{proof}

\begin{proposition}\label{prop:equiv} Suppose $\mu \in \mathfrak{M}_{x}(\mathcal{U})$ and $\mu$ is definable over $M$. Then $r_{\mu} = s_{\mu}$. 
\end{proposition} 

\begin{proof}  Fix $\varphi(x,y) \in \mathcal{L}_{xy}$ and $h \in \mathbfcal{C}$. Fix $(h_i, \mathcal{A}_{i})_{i \in I}$ such that for each $i \in I$ we have $h_i \in (M^{\Omega}_0)^{y}$, $\mathcal{A}_{i}$ is partition for $h_i$, and $\lim_{i \in I} \tp(h_i/\mathbfcal{M}^{\Omega}) = \tp(h/\mathbfcal{M}^{\Omega})$. Consider the following computation:
\begin{align*} 
(\mathbb{E}[\varphi(x,h)])^{r_{\mu}} &= F_{r_{\mu}}^{\mathbb{E}[\varphi]}(\tp(h/\mathbfcal{M}^{\Omega}))\\
&\overset{(a)}{=}\lim_{i \in I} F_{r_{\mu}}^{\mathbb{E}[\varphi]}(\tp(h_i/\mathbfcal{M}^{\Omega}))  \\
&=\lim_{i \in I} (\mathbb{E}[\varphi(x,h_i)])^{r_\mu}\\
&= \lim_{i \in I} \sum_{A \in \mathcal{A}_i} \mathbb{P}_0(A) \mu(\varphi(x,h_i|_{A}))\\
&\overset{(b)}{=} \lim_{i \in I} \int_{S_{y}(M)} F_{\mu}^{\varphi} d \left( \sum_{A \in \mathcal{A}_i} \mathbb{P}_0(A) \delta_{(h_i|_{A})} \right) \\
&\overset{(c)}{=}  \lim_{i \in I} \int_{S_{y}(M)} F_{\mu}^{\varphi} d\left(\nu_{\tp(h_i/\mathbfcal{C})}\right)\\
&\overset{(d)}{=}   \int_{S_{y}(M)} F_{\mu}^{\varphi} d\left(\lim_{i \in I}\nu_{\tp(h_i/\mathbfcal{C})} \right)\\
&\overset{(e)}{=}   \int_{S_{y}(M)} F_{\mu}^{\varphi} d \left(\nu_{\lim_{i \in I}\tp(h_i/\mathbfcal{C})} \right)\\
&\overset{(f)}{=}   \int_{S_{y}(M)} F_{\mu}^{\varphi} d\left(\nu_{\tp(h/\mathbfcal{C})}\right)\\
&= (\mathbb{E}[\varphi(x,h)])^{s_{\mu}}.  
\end{align*} 
By quantifier elimination, the statement holds. We provide the following justifications:
\begin{enumerate}[(a)]
\item Since $r_{\mu}$ is definable, the map $F_{r_{\mu}}^{\mathbb{E}[\varphi]}: S_{y}(\mathbfcal{M}^{\Omega}) \to [0,1]$ is continuous and so commutes with nets. 
\item For fixed $i$, the value of the left-hand-side of the equality is precisely the value of the right-hand-side of the equality. 
\item Follows directly from Fact \ref{Fact:easy-1}. 
\item Since $\mu$ is definable, the map $\int_{S_{y}(M)} F_{\mu}^{\varphi} d- :\mathfrak{M}_{y}(\mathcal{U}) \to [0,1]$ is continuous and thus commutes with nets. 
\item The map $\nu_{-}:S_{y}(\mathbfcal{C}) \to \mathfrak{M}_{y}(\mathcal{U})$ is continuous and so commutes with nets. 
\item Hypothesis for choice of $h_i$'s. \qedhere
\end{enumerate} 
\end{proof} 

In \cite{CGH2}, we (along with Conant and Hanson) verified that if $\mu$ is definable, then the type $r_{\mu}$ is a consistent (\cite[Fact 3.10]{CGH2}). In the following, we use a simpler version of this statement, i.e. that if $a_1,..,a_n$ is a sequence of elements in $M$, then the type measure $r_{Av(\bar{a})}$ is consistent. We provide a quick proof:

\begin{fact}\label{fact:con} Suppose that $a_1,...,a_n$ is a sequence of elements in $M$. Then the type $r_{\Av(\bar{a})}$ is consistent. 
\end{fact}

\begin{proof} We give a sketch. By quantifier elimination, we have $\mathbfcal{U}^{\Omega} \prec \mathbfcal{U}^{\Omega \times [0,1)}$ (see \cite[Proposition 3.11]{berenstein2021definable} for direct proof). More explicitly, our probability space is $(\Omega_0 \times [0,1), \mathcal{B}_0 \times \mathcal{B}_{[0,1)},\mathbb{P}_{0} \times L)$ where $L$ is the Lebesgue measure and $\mathcal{B}_{[0,1)}$ is the collection of Borel subsets of $[0,1)$. Consider the element $h_{\bar{a}} \in \mathcal{U}^{\mathcal{B} \times [0,1]}$ where
$h_{\bar{a}}(t,s) = a_j$ 
whenever $s \in [\frac{j-1}{n},\frac{j}{n})$. Let $q = \tp(h_{\bar{a}}/\mathcal{U}^{\Omega})$. We claim that $q$ is definable over $\mathcal{U}^{\Omega}$. Indeed, for any $\mathcal{L}$-formula $\varphi(x,y)$ and any $b \in \mathcal{U}_0^{\Omega}$, we have the following computation:  
\begin{align*}
    (\mathbb{E}[\varphi(x,b)])^{q} &= (\mathbb{E}[\varphi(h_{\bar{a}},b)]) \\ &= \mathbb{P}_0 \times L  \left(\left\{ (t,s) \in \Omega_0 \times [0,1) : \mathcal{U} \models \varphi(h_{\bar{a}}(t,s), b(t,s)) \right\} \right) \\
    &=\sum_{i=1}^{n} \mathbb{P}_0 \times L  \left(\left\{ (t,s) \in \Omega_0 \times \left[\frac{i-1}{n},\frac{i}{n} \right) : \mathcal{U} \models \varphi(h_{\bar{a}}(t,s), b(t,s)) \right\} \right) \\
    &=\sum_{i=1}^{n} \mathbb{P}_0 \times L  \left(\left\{ (t,s) \in \Omega_0 \times \left[\frac{i-1}{n},\frac{i}{n} \right) : \mathcal{U} \models \varphi(a_i, b(t)) \right\} \right) \\
    &=\sum_{i=1}^{n} \frac{1}{n} \mathbb{P}_0  \left(\left\{ t \in \Omega_0 : \mathcal{U} \models \varphi(a_i, b(t) \right\} \right) \\
    &=\sum_{i=1}^{n} \frac{1}{n} \mathbb{E}[\varphi(f_a,b)]. 
\end{align*}
Let $\hat{q}$ be the unique definable extension of $q$ to $S_{x}(\mathbfcal{C})$. We claim that $\hat{q} = r_{\Av(\bar{a})}$. 
\end{proof}

\begin{proposition}[T NIP]\label{prop:consistent} For any $\mu \in \mathfrak{M}_{x}^{\fs}(\mathcal{U},M)$, the type $s_{\mu}$ is consistent and finitely satisfiable in $\mathbfcal{M}^{\Omega}$. 
\end{proposition}

\begin{proof}
 Fix $\mathcal{L}$-formulas $\{\varphi_{i}(x,y_i)\}_{i=1}^{n}$ and $b_1,...,b_n$ from $\mathbfcal{C}$, and $\epsilon > 0$. By quantifier elimination, it suffices to show that there exists a type $q \in S_{x}(\mathbfcal{C})$ such that
\begin{equation*}
    |(\mathbb{E}[\varphi_i(x,b_i)])^{q} - (\mathbb{E}[\varphi_i(x,b_i)])^{s_{\mu}}| < \epsilon. 
\end{equation*}
Since $\mu$ is finitely satisfiable in $M$, there exists a net of measures $(\mu_j)_{j \in J}$ such that $\mu_j = \Av(\bar{a}_{j}) \in \conv(M)$ and $\lim_{j \in J} \mu_{j} = \mu$.  
Notice that for each $i \leq n$, 
\begin{align*}
    (\mathbb{E}[\varphi_i(x,b_i)])^{s_{\mu}} &= \int_{S_y(M)} F_{\mu}^{\varphi_i} d\nu_{\tp(b_i/\mathbfcal{C})} \\ &= (\mu \otimes \nu_{\tp(b_i/\mathbfcal{C})})(\varphi_i(x,y)) \\ & \overset{(*)}{=} \lim_{j \in J} (\Av(\bar{a}_i) \otimes \nu)(\varphi_i(x,y)) \\ &= \lim_{j \in J} \int_{S_y(M)} F_{\Av(\bar{a}_j)}^{\varphi_i} d\nu_{\tp(b/\mathbfcal{C})} \\ 
    &= \lim_{j \in J} (\mathbb{E}[\varphi_i(x,b_i)])^{s_{\Av(\bar{a}_j)}}. 
\end{align*}
Equation $(*)$ holds by Remark \ref{remark:tools}. Since for each $i \leq n$, the map $\mathbb{E}[\varphi_i(x,b_i)]: S_{x}(\mathbfcal{C}) \to [0,1]$ via $p \to (\mathbb{E}[\varphi_i(x,b_i)])^{p}$ is continuous, we can find some $j \in J$ such that for each $i \leq n$, 
\begin{equation*}
    |\mathbb{E}[\varphi_i(x,b_i)]^{s_{\Av(\bar{a}_j)}} - \mathbb{E}[\varphi_i(x,b_i)]^{s_{\mu}}| < \epsilon.  
\end{equation*}
Since $s_{\Av(\bar{a}_j)} = r_{\Av(\bar{a}_j)}$ (by Proposition \ref{prop:equiv}) and $r_{\Av(\bar{a}_j)}$ is a consistent type (by Fact \ref{fact:con}), the type $s_{\mu}$ is consistent. 

Moreover, for each $j \in J$, the type $r_{\Av(\overline{a}_j)}$ is generically stable over $\mathbfcal{M}^{\Omega}$ (see \cite[Corollary 3.19]{CGH2}) and this implies that $r_{\Av(\overline{a}_j)}$ is finitely satisfiable in $\mathbfcal{M}^{\Omega}$. Since the collection of types which are finitely satisfiable in $\mathbfcal{M}^{\Omega}$ is a closed set and $\lim_{j \in J} s_{\Av(\bar{a}_j)} = s_{\mu}$, we conclude that $s_{\mu}$ is finitely satisfiable in $\mathbfcal{M}^{\Omega}$. 
\end{proof}

\begin{proposition}\label{prop:cont-s} If $T$ is NIP, then the map $s_{-}:\mathfrak{M}_{x}^{\fs}(\mathcal{U},M) \to S_{x}^{\fs}(\mathbfcal{C},\mathbfcal{M}^{\Omega})$ is continuous. 
\end{proposition}
\begin{proof} Similar to the proof above. Let $(\mu_i)_{i \in I}$ be a net of elements in $\mathfrak{M}_{x}^{\fs}(\mathcal{U},M)$ such that $\lim_{i \in I} \mu_i = \mu$. It suffices to prove that $\lim_{i \in I} s_{\mu_i} = s_{\mu}$. By quantifier elimination, it suffices to show these terms agree on all atomic formulas. Let $\varphi(x,y)$ be an $\mathcal{L}$-formula and $b \in \mathbfcal{C}$. Then 
\begin{align*}
    (\mathbb{E}[\varphi(x,b)])^{s_{\mu}} &= \int_{S_{y}(M)} F_{\mu}^{\varphi} d\nu_{\tp(b/\mathbfcal{C})} \\
    &= (\mu \otimes \nu_{\tp(b/\mathbfcal{C})})(\varphi(x,y)) \\
    &= ((\lim_{i \in I} \mu_i) \otimes \nu_{\tp(b/\mathbfcal{C})})(\varphi(x,y))\\
    &\overset{(*)}{=} \lim_{i \in I} (\mu_i \otimes \nu_{\tp(b/\mathbfcal{C})})(\varphi(x,y)) \\
    &=\lim_{i \in I} \int_{S_{x}(M)} F_{\mu_i}^{\varphi} d\nu_{\tp(b/\mathbfcal{C})} \\
    &=\lim_{i \in I} (\mathbb{E}[\varphi(x,b)])^{s_{\mu_i}}
\end{align*}
We remark that Equation $(*)$ follows from Remark \ref{remark:tools}.
\end{proof}

The next proposition is useful in proving that the restriction of a definable type in the randomization yields a definable measure, but seems slightly important in its own right. 

\begin{proposition}\label{prop:cont} Consider the map $f: S_{x}(M) \to S_{x}(\mathbfcal{M}^{\Omega})$ where if $a \models q$, then $q \to \tp(f_{a}/\mathbfcal{M}^{\Omega})$. This map is well-defined and continuous. 
\end{proposition}

\begin{proof} Well-definedness follows quickly from an automorphism argument. Let $a,b \models p$ and $\sigma \in \Aut( \mathcal{U}/M)$ map $a \to b$. Then $\sigma$ extends to an automorphism $\sigma' \in \Aut(\mathbfcal{U}^{\Omega}/\mathbfcal{M}^{\Omega})$ which maps $f_{a}$ to $f_{b}$. More explicitly, $\sigma'$ is the identity on the probability algebra sort and for any $h \in \mathcal{U}^{\Omega}_0$, we have that $\sigma'(h) = \sigma \circ h$. $\sigma'$ is extended to $\mathcal{U}^{\Omega}$ via continuity. 

We now argue that $f$ is continuous. Suppose that $\lim_{i \in I} p_i = p$. Let $a_i \models p_i$ and $a \models p$. It suffices to prove that for each formula $\varphi(x,y) \in \mathcal{L}$ and $h \in M^{\Omega}_0$ that $\lim_{i \in I} \mathbb{E}[\varphi(x,h)]^{\tp(f_{a_i}/M)} = \mathbb{E}[\varphi(x,h)]^{\tp(f_{a}/M)}$. Fix a partition $\mathcal{A}$ for $h$. Now notice that 
\begin{align*}
    \lim_{i \in I} \mathbb{E}[\varphi(x,h)]^{\tp(f_{a_i}/\mathbfcal{M}^{\Omega})} &= \lim_{i \in I} \mathbb{E}[\varphi(f_{a_i},h)] \\
    &= \lim_{i \in I}  \mathbb{P}_0(\{t \in \Omega_0: \mathcal{U} \models \varphi(f_{a_i}(t), h(t))\} \\
    &= \lim_{i \in I}  \sum_{A \in \mathcal{A}} \mathbb{P}_0(\{t \in A: \mathcal{U} \models \varphi(a_i, h(t))\} \\
    &= \sum_{A \in \mathcal{A}} \mathbb{P}(A) \lim_{i \in I} \mathbf{1}_{\varphi(x,y)}(a_i,h|_{A}) \\
    &= \sum_{A \in \mathcal{A}} \mathbb{P}(A) \mathbf{1}_{\varphi(x,y)}(a,h|_{A}) \\
    &= \mathbb{P}_0 (\{t \in \Omega_0: \mathcal{U} \models \varphi(f_{a}(t),h(t))\})  \\
    &= (\mathbb{E}{[\varphi(x,h)]})^{\tp(f_a/\mathbfcal{M}^{\Omega})}. \qedhere 
\end{align*}
\end{proof}

\begin{proposition} Suppose $q \in S_{x}(\mathbfcal{C})$. 
\begin{enumerate} 
\item if $q$ is $\mathbfcal{M}^{\Omega}$-invariant, then $\nu_{q}$ is $M$-invariant. 
\item if $q$ is finitely satisfiable in $\mathbfcal{M}^{\Omega}$, then $\nu_{q}$ is finitely satisfiable in $M$. 
\item if $q$ if $\mathbfcal{M}^{\Omega}$-definable, then $\nu_{q}$ is $M$-definable.
\item if $q$ is \emph{fam} over $\mathbfcal{M}^{\Omega}$, then $\nu_{q}$ is \emph{fam} over $M$. 
\end{enumerate} 
\end{proposition} 

\begin{proof} We prove the statements: 
\begin{enumerate} 
\item Let $a,b \in \mathcal{U}^{y}$ and suppose that $\tp(a/M) = \tp(b/M)$. Then there exists an automorphism $\sigma: \mathcal{U} \to \mathcal{U}$ which fixes $M$ and maps $a \to b$. Again, we claim that $\sigma':\mathcal{U}_0^{\Omega} \to \mathcal{U}_0^{\Omega}$ via $\sigma'(h) = \sigma \circ h$ is an automorphism of $\mathbfcal{U}^{\Omega}$ fixing $\mathbfcal{M}^{\Omega}$. Now
\begin{align*} 
\nu_{q}(\varphi(x,a)) = (\mathbb{E}[\varphi(x,f_{a})])^{q} &= (\mathbb{E}[\varphi(x,\sigma'(f_a))])^{q} \\ &= (\mathbb{E}[\varphi(x,f_{b})])^{q} = \nu_{q}(\varphi(x,b)). 
\end{align*} 
\item Suppose that $\nu_{q}(\varphi(x,b)) > 0$. Then $(\mathbb{E}[\varphi(x,f_{b})])^{q} > 0$. So, there exists some $h \in M_0^{\Omega}$ such that $(\mathbb{E}[\varphi(h,f_{b})]) > 0$. If $\mathcal{A}$ is a partition for $h$, then  
\begin{align*}
0< (\mathbb{E}[\varphi(h,f_{b})])  &= \mathbb{P}_0 \Big( \Big\{ t \in \Omega_0: \mathcal{U} \models \varphi(h(t_*),b)  \Big\} \Big). \\ &=\sum_{A \in \mathcal{A}}\mathbb{P}_0(A)\mathbf{1}_{\varphi(x,y)}(h|_{A},b) 
\end{align*}  
So there is some $A_* \in \mathcal{A}$ such that $\mathbb{P}_{0}(A_*) > 0$. Consider $h|_{A_*}$. 
\item Fix an $\mathcal{L}$-formula $\varphi(x,y)$. Consider $(p_i)_{i \in I}$ and $p$ in $S_{y}(M)$ such that $\lim_{i \in I} p_i = p$. It suffices to show that $\lim_{i \in I} F_{\nu_{q}}^{\varphi}(p_i) = F_{\nu_{q}}^{\varphi}(p)$.
Fix $(c_i)_{i \in I}$ and $c$ in $\mathcal{U}$ such that $c_i \models p_i$ and $c \models p$. By Fact \ref{prop:cont},
$\lim_{i \in I} \tp(f_{c_i} / \mathbfcal{M}^\Omega) = \tp(f_{c}/\mathbfcal{M}^{\Omega})$. Now notice that 
\begin{align*}
    F_{\nu_{q}}^{\varphi}(q) &= \nu_{q}(\varphi(x,c)) = (\mathbb{E}[\varphi(x,f_{c})])^{q} = F_{q}^{\mathbb{E}[\varphi(x,y)]}(\tp(f_c/\mathbfcal{M}^{\Omega})) \\ &\overset{(*)}{=} \lim_{i \in I} F_{q}^{\mathbb{E}[\varphi(x,y)]}(\tp(f_{c_i}/\mathbfcal{M}^{\Omega})) = \lim_{i \in I} (\mathbb{E}[\varphi(x,f_{c_i})]^{q} \\ &= \lim_{i \in I} \nu(\varphi(x,c_i)) = \lim_{i \in I} F_{\nu_{q}}^{\varphi}(q_i). 
\end{align*}
Where equation $(*)$ follows from definability of $q$. 
\item Fix an $\mathcal{L}$-formula $\varphi(x,y)$ and $\epsilon > 0$. Since $p$ is \emph{fam} over $\mathbfcal{M}^{\Omega}$, there exists $h_1,...,h_n \in M^{\Omega}_0$ such that 
\begin{equation*}
    \sup_{g \in \mathbfcal{C}} |(\mathbb{E}[\varphi(x,g)])^{p} - \frac{1}{n}\sum_{i \leq n} \mathbb{E}[\varphi(h_i,g)] | < \epsilon. 
\end{equation*}
For each $i \leq n$, choose a partition $\mathcal{A}_i$ for $h_i$. Now
\begin{align*}
    \nu_{p}(\varphi(x,b)) &= (\mathbb{E}[\varphi(x,f_b)])^{p}\\
    &\approx_{\epsilon} \frac{1}{n}\sum_{i \leq n} \mathbb{E}[\varphi(h_i,f_{b})] \\
    &= \frac{1}{n} \sum_{i \leq n} \mathbb{P}_0\left( \{t \in \Omega_0: \mathcal{U} \models \varphi(h_i(t),b) \} \right) \\
    &= \frac{1}{n} \sum_{i \leq n} \sum_{A \in \mathcal{A}_i} \mathbb{P}_0(A) \mathbf{1}_{\varphi(x,y)}(h_i|_{A},b) \\
    &= \left( \frac{1}{n}\sum_{i \leq n}\sum_{A \in \mathcal{A}_i} \mathbb{P}_0(A)\delta_{h_i|_{A}} \right)  (\varphi(x,b)). 
\end{align*}
The final term can be approximated by a measure of the form $\Av(\overline{c})$ where each $c_i$ appears in $\bigcup_{i=1}^{n}\{h_i|_{A}: A \in \mathcal{A}_i\}$. Indeed, the only issue is that $\mathbb{P}(A)$ may be irrational and so choosing a small enough rational approximation completes the proof. \qedhere
\end{enumerate} 
\end{proof} 

The following question remains open. 

\begin{question} Let $q \in S^{\fs}_{x}(\mathbfcal{C},\mathbfcal{M}^{\Omega})$. Suppose that $q$ is \emph{generically stable} over $\mathbfcal{M}^{\Omega}$. Does this imply that $\nu_{q}$ is fim? We remark that if the answer is `Yes', this answers an open question from \cite{CGH2}. 
\end{question}

\begin{proposition}\label{prop:basic0} Let $\mu \in \mathfrak{M}^{\fs}_{x}(\mathcal{U},M)$. Then $\nu_{s_{\mu}} = \mu$. 
\end{proposition} 

\begin{proof} Fix $\varphi(x,y) \in \mathcal{L}_{xy}$ and $b \in \mathcal{U}^{y}$. Then
\begin{align*} 
\nu_{s_{\mu}}(\varphi(x,b)) &= (\mathbb{E}[\varphi(x,f_{b})])^{s_{\mu}}\\
&=\int_{S_{x}(M)} F_{\mu}^{\varphi} d \left( \nu_{\tp(f_{b}/\mathbfcal{C})} \right)\\
&=\int_{S_{x}(M)} F_{\mu}^{\varphi} d \left( \delta_{\tp(b/M)} \right)\\
&= F_{\mu}^{\varphi}(\tp(b/M))\\
&= \mu(\varphi(x,b)). \qedhere
\end{align*} 
\end{proof} 

\begin{warning} We remark that in general,  $s_{\nu_{q}} \neq q$. In particular, the map $\nu_{-}$ is far from being injective. Fix any $A \subseteq \Omega_0$ such that $\mathbb{P}_0(A) = 1/2$ and $b_1,b_2 \in M$. Consider the random element $g:\Omega \to \mathcal{U}$ where
\begin{equation*}
g(t)=\begin{cases}
\begin{array}{cc}
b_{1} & t\in A,\\
b_{2} & t\in A^{c}.
\end{array}\end{cases}
\end{equation*} 
Notice that $g \in M^{\Omega} \prec \mathbfcal{C}$. For any choice of such $A$, $\nu_{\tp(g/\mathbfcal{C})} = \frac{1}{2}\delta_{b_1} + \frac{1}{2}\delta_{b_2}$. 
\end{warning}

\section{Going up: Extension-by-definition commutes with the Morley product} 
We show that if $T$ is NIP and the global measures $\mu$ and $\nu$ that are finitely satisfiable in $M$, then $s_{\mu} \otimes s_{\nu} = s_{\mu \otimes \nu}$. We remark that the main proof in this section can also be easily repurposed to give another proof of $r_{\mu} \otimes r_{\nu} = r_{\mu \otimes \nu}$ for definable measures in arbitrary theories (see Remark \ref{remark:def-arb}). This comes down to swapping out one justification for another, i.e. the fact that ``the Morley product of invariant measures over NIP theories is associative" replaced with ``the Morley product of definable measures over arbitrary theories is associative". The original proof of the definable case is \cite[Proposition 3.15]{CGH2}, yet the proof presented here is quite different. The next two results are straightforward, but we provide the proofs for clarity. 

\begin{proposition}\label{prop:basic1} Suppose that $b \in \mathbfcal{C}^{x}$ and $a \in \mathcal{U}^{y}$. Then $\nu_{\tp(b/\mathbfcal{C})} \otimes \delta_{a} = \nu_{\tp(b,f_{a}/\mathbfcal{C})}$. 
\end{proposition} 

\begin{proof} Fix a formula $\varphi(x,y,z) \in \mathcal{L}_{xyz}$ and $d \in \mathcal{U}^{z}$. Then we notice; 
\begin{align*}
\left( \nu_{\tp(b/\mathbfcal{C})} \otimes \delta_{a} \right) (\varphi(x,y,d)) &=
  \nu_{\tp(b/\mathbfcal{C})}(\varphi(x,a,d))\\
&= \mathbb{E}[\varphi(x,f_{a},f_{d})]^{\tp(b/\mathbfcal{C})}\\
&= \mathbb{E}[\varphi(x,y,f_{d})]^{\tp(b,f_{a}/\mathbfcal{C})}\\
&= \nu_{\tp(b,f_{a}/\mathbfcal{C})}(\varphi(x,y,d)). \qedhere
\end{align*} 
\end{proof} 

\begin{proposition}\label{prop:basic2} Suppose that $\lambda \in \mathfrak{M}_{x}(\mathcal{U})$ and $c \models s_{\lambda}|_{\mathbfcal{U}^{\Omega}}$. Then $\nu_{\tp(c/\mathbfcal{C})} = \lambda$.
\end{proposition} 
\begin{proof}
Fix $\varphi(x,y) \in \mathcal{L}_{xy}$ and $d \in \mathcal{U}^{y}$. Consider 
\begin{align*}
\nu_{\tp(c/\mathbfcal{C})}(\varphi(x,d)) &= (\mathbb{E}[\varphi(x,f_{d})])^{\tp(c/\mathbfcal{C})}\\
&\overset{(a)}{=} (\mathbb{E}[\varphi(x,f_{d})])^{s_{\lambda}}\\
&= \nu_{s_{\lambda}}(\varphi(x,d))\\
&\overset{(b)}{=} \lambda(\varphi(x,d)). 
\end{align*} 
We provide the following justifications: 
\begin{enumerate}[(a)]
\item Since $c \models s_{\mu}|_{\mathbfcal{U}^{\Omega}}$ and $\mathbb{E}[\varphi(x,f_{d})] \in \mathcal{L}_{x}^{R}(\mathcal{U}^{\Omega})$. 
\item Proposition \ref{prop:basic0}. \qedhere
\end{enumerate}
\end{proof}

\begin{theorem}[T NIP]\label{prop:main1} Suppose $\mu \in \mathfrak{M}_{x}^{\fs}(\mathcal{U},M)$ and $\lambda \in \mathfrak{M}_{y}^{\fs}(\mathcal{U},M)$. Then $s_{\mu} \otimes s_{\lambda} = s_{\mu \otimes \lambda}$. 
\end{theorem} 
\begin{proof} Fix $\varphi(x,y,z) \in \mathcal{L}_{xyz}$ and $b \in \mathbfcal{C}^{z}$. Let $c \models s_{\lambda}|_{\mathbfcal{U}^{\Omega}b}$.  Now consider the following computation:
\begin{align*}
(\mathbb{E}[\varphi(x,y,b)])^{s_{\mu} \otimes s_{\lambda}} &= (\mathbb{E}[\varphi(x,c,b)])^{s_{\mu}}\\ 
&= \int_{S_{yz}(M)} F_{\mu}^{\varphi} d \left( \nu_{\tp(c,b/\mathbfcal{C})} \right)\\
&\overset{(a)}{=} \int_{S_{yz}(M)} F_{\mu}^{\varphi} d \left( (\nu_{\tp(c/\mathbfcal{C})} \otimes \nu_{\tp(b/\mathbfcal{C})}) \right)\\
&= \left( \mu \otimes \left(\nu_{\tp(c/\mathbfcal{C})} \otimes \nu_{\tp(b/\mathbfcal{C})} \right) \right)(\varphi(x,y,z)) \\
&\overset{(b)}{=} \left( \mu \otimes \left(\lambda \otimes \nu_{\tp(b/\mathbfcal{C})} \right) \right)(\varphi(x,y,z)) \\
&\overset{(c)}{=} \left(\left(\mu \otimes \lambda \right) \otimes \nu_{\tp(b/\mathbfcal{C})}\right)(\varphi(x,y,z)) \\
&= \int_{S_{z}(M)} F_{\mu \otimes \lambda}^{\varphi} d( \nu_{\tp(b/\mathbfcal{C})}) \\
&= (\mathbb{E}[\varphi(x,y,b)])^{s_{\mu \otimes \lambda}}. \\
\end{align*} 
We provide the following justifications: 
\begin{enumerate}[(a)]
\item Notice $(\nu_{\tp(c/\mathbfcal{C})} \otimes \nu_{\tp(b/\mathbfcal{C})})|_{M} = \nu_{\tp(c,b/\mathbfcal{C})}|_{M}$. Indeed, if $\varphi(y,z,w) \in \mathcal{L}_{yzw}$ and $a \in M^{w}$, then 
\begin{align*}
\left( \nu_{\tp(c/\mathbfcal{C})} \otimes \nu_{\tp(b/\mathbfcal{C})} \right)(\varphi(y,z,a))  &\overset{(i)}{=} \left( \lambda \otimes \nu_{\tp(b/\mathbfcal{C})} \right)(\varphi(y,z,a))\\
&= \int_{S_{z}(M)} F_{\lambda}^{\varphi_{a}} d \left( \nu_{\tp(b/\mathbfcal{C})} \right)\\
&= \int_{S_{zw}(M)} F_{\lambda}^{\varphi} d \left(\nu_{\tp(b/\mathbfcal{C})} \otimes \delta_{\tp(a/\mathcal{U})} \right)\\
&\overset{(ii)}{=} \int_{S_{zw}(M)} F_{\lambda}^{\varphi} d \left(\nu_{\tp(b,f_{a}/\mathbfcal{C})} \right)\\
&=(\mathbb{E}[\varphi(y,b,f_{a})])^{s_{\lambda}}\\
&\overset{(iii)}{=}(\mathbb{E}[\varphi(y,b,f_{a})])^{\tp(c/\mathbfcal{C})}\\
&=(\mathbb{E}[\varphi(y,z,f_{a})])^{\tp(c,b/\mathbfcal{C})}\\
&=\nu_{\tp(c,b/\mathbfcal{C})}(\varphi(y,z,a)).
\end{align*}  
We provide the following mini-justifications: 
\begin{enumerate}[(i)]
\item Proposition \ref{prop:basic2}. 
\item Proposition \ref{prop:basic1}.
\item Since $c \models s_{\lambda}|_{\mathbfcal{U}^{\Omega}b}$. 
\end{enumerate} 
\item Proposition \ref{prop:basic2}.
\item Associativity of the Morley product in NIP theories. \qedhere
\end{enumerate} 
\end{proof}

\begin{remark}\label{remark:def-arb} Theorem \ref{prop:main1} and along with Proposition \ref{prop:equiv} give another proof of the following: Suppose that $\mu \in \mathfrak{M}_{x}^{\inv}(\mathcal{U},M)$, $\nu \in \mathfrak{M}_{x}^{\inv}(\mathcal{U},M)$, and both $\mu$ and $\nu$ are definable. Then
\begin{equation*}
    r_{\mu}(x) \otimes r_{\nu}(y) = r_{\mu \otimes \nu}(x,y). 
\end{equation*}
This was originally proved in \cite[Proposition 3.15]{CGH2}. The only difference in the proof is the justification of $(c)$. Here ``associativity of the Morley product in NIP theories" is replaced by ``associativity of the Morley product for definable measures in arbitrary theories" (see Remark \ref{remark:tools}).
\end{remark}

\begin{remark} If the types $\{s_{\mu} : \mu \in \mathfrak{M}_{x}^{\inv}(\mathcal{U},M)\}$ are consistent, then the proof of Theorem \ref{prop:main1} directly generalizes to this case as well. 
\end{remark}

\section{Going down: The restriction map and the Morley product} 

We warn the reader that the restriction map does not interact well with arbitrary Morley products.

\begin{warning} We remark that in general,  $\nu_{p} \otimes \nu_{q} \neq \nu_{p \otimes q}$. Fix any $A \subseteq \Omega_0$ such that $\mathbb{P}_0(A) = 1/2$ and $b_1,b_2 \in M$. Consider the random elements $g_1,g_2:\Omega \to \mathcal{U}$ where
\begin{equation*}
g_1(t)=\begin{cases}
\begin{array}{cc}
b_{1} & t\in A,\\
b_{2} & t\in A^{c},
\end{array}\end{cases} 
\text{and } g_2(t)=\begin{cases}
\begin{array}{cc}
b_{2} & t\in A,\\
b_{1} & t\in A^{c}.
\end{array}\end{cases} 
\end{equation*} 
Let $p = \tp(g_1 /\mathbfcal{C})$ and $q = \tp(g_2/\mathbfcal{C})$. 
\begin{equation*}
    \nu_{p \otimes q}(x = y) = 0 \text{ and } (\nu_{p} \otimes \nu_{q}) (x = y) = 1/2. 
\end{equation*}
\end{warning} 

This warning shows us that we cannot expect $\nu_{p \otimes q} = \nu_{p} \otimes \nu_{q}$. However, we would like a condition which implies that the Morley product commutes with the restriction map. It turns our that one can give a characterization of this property, for finitely satisfiable types, by writing down a family of equalities. This condition is a kind of \emph{strong probabilistically independent} which we call \emph{prob. independence}. We then prove that in the NIP setting, $p$ is prob. independent over $q$ if and only if $\nu_{p \otimes q} = \nu_{p} \otimes \nu_{q}$. 

\begin{definition}[Prob. independence] Let $h \in (\mathcal{U}_0^{\Omega})^{x}$ and $b \in {\mathbfcal{C}}^{y}$. We say that $h$ is prob. independent over $b$ if for any partition $\mathcal{A}$ for $h$, $\mathcal{L}$-formulas $\theta(x,y,z)$, and points $e \in \mathcal{U}^{z}$, we have that
\begin{equation*}
    \mathbb{E}[\theta(h,b,f_{e})] = \sum_{A \in \mathcal{A}} \mathbb{P}_0(A) \mathbb{E}[\theta(f_{h|_A},b,f_{e})]. 
\end{equation*}
If $p \in S_{x}^{\fs}(\mathbfcal{C},\mathbfcal{M}^{\Omega})$, we say that $p$ is prob. independent over $b$ if there exists a net $(h_i)_{i \in I}$ of points in $M_{0}^{\Omega}$ such that $\lim_{i \in I} \tp(h_i/\mathbfcal{C}) = p $, and for any $\mathcal{L}$-formula $\theta(x,y,z)$ and parameters $e$ from $\mathcal{U}^{z}$, 
\begin{equation*}
    \mathbb{E}[\theta(x,b,f_{e})]^{p} = \lim_{i \in I} \sum_{A \in \mathcal{A}_i} \mathbb{P}_0(A)\mathbb{E}[\theta(f_{h_i|_A},b,f_e)]. 
\end{equation*}
where $\mathcal{A}_i$ is a partition for $h_i$. 

Finally, we say that a type $p \in S_{x}^{\fs}(\mathbfcal{C},\mathbfcal{M}^{\Omega})$ is prob. independent over a type $q \in S_{y}^{\fs}(\mathbfcal{C},\mathbfcal{M}^{\Omega})$ if for some/any $b \models q|_{\mathbfcal{M}^{\Omega}}$ we have that $p$ is prob. independent over $b$. 
\end{definition}

We now give three examples of different kinds of random variables which satisfy the condition above. Our first example deals with constant random variables, the second with \emph{extensions-by-definition}, and the final one is constructed by hand.

\begin{fact} For any $e \in M$, and $b \in \mathbfcal{C}$, we have that
\begin{enumerate}
    \item $e$ is prob. independent over $b$. 
    \item $\tp(f_{e}/\mathbfcal{C})$ is prob. independent over $b$. 
    \item For any $a \in M$, $\tp(f_{e}/\mathbfcal{C})$ is prob. independent over $\tp(f_{a}/\mathbfcal{C})$. 
\end{enumerate}
\end{fact}

\begin{proposition}[T NIP] If $\mu \in \mathfrak{M}_{x}^{\fs}(\mathcal{U},M)$ and $b \in \mathbfcal{C}^{y}$, then $s_{\mu}$ is prob. independent over $b$. This implies that for any $q \in S_{y}^{\fs}(\mathbfcal{C},\mathbfcal{M}^{\Omega})$, $s_{\mu}$ is prob. independent over $q$.
\end{proposition}

\begin{proof} By Proposition \ref{prop:consistent}, we know that $s_{\mu}$ is finitely satisfiable in $\mathbfcal{M}^{\Omega}$ and so the proposition itself makes sense. Fix an $\mathcal{L}$-formula $\varphi(x,y,z)$, parameters $e \in \mathcal{U}^{z}$, and a net $(h_i)_{i \in I}$ such that $\lim_{i \in I} \tp(h_i/\mathbfcal{C}) = s_{\mu}$. For each $i \in I$, we let $\mathcal{A}_i$ be a partition for $h_i$. Notice that $\lim_{i \in I} \nu_{\tp(h_i/\mathbfcal{C})} = \mu$ because 

\begin{equation*}
\mu = \nu_{s_{\mu}} = \nu_{\lim_{i \in I} \tp(h_i/ \mathbfcal{C})} = \lim_{i \in I} \nu_{\tp(h_i/\mathbfcal{C})}. 
\end{equation*}
where the first equality follows from Proposition \ref{prop:basic0} and the last equality follows from continuity of $\nu_{-}$ (Proposition \ref{prop:cont-v}). Now we compute: 
\begin{align*}
(\mathbb{E}[\varphi(x,b,f_{e})])^{s_{\mu}} &= \int_{S_{yz}(M)} F_{\mu}^{\varphi(x;yz)} d\nu_{\tp(b,f_e/\mathbfcal{C})} \\
&= (\mu \otimes \nu_{\tp(b/\mathbfcal{C})})(\varphi(x,y,e))\\
&= ((\lim_{i \in I} \nu_{\tp(h_i/\mathbfcal{C})}) \otimes \nu_{\tp(b/\mathbfcal{C})})(\varphi(x,y,e)) \\
&\overset{(*)}{=} \lim_{i \in I} ((\nu_{\tp(h_i/\mathbfcal{C})} \otimes \nu_{\tp(b/\mathbfcal{C})})(\varphi(x,y,e))) \\
&\overset{(**)}{=} \lim_{i \in I} \left( \left(  \sum_{A \in \mathcal{A}_i} \mathbb{P}_0(A) \delta_{h_i|_{A}} \right)  \otimes \nu_{\tp(b/\mathbfcal{C})}(\varphi(x,y,e)) \right) \\
&= \lim_{i \in I} \sum_{A \in \mathcal{A}_i} \mathbb{P}_0(A)\nu_{\tp(b/\mathbfcal{C})} \varphi(h_i|_{A},y,e)) \\
&=\lim_{i \in I} \sum_{A \in \mathcal{A}_i} \mathbb{P}(A)\mathbb{E}[\theta(f_{h_i|_A},b,f_e)]. 
\end{align*}
We remark that Equation $(*)$ follows from left-continuity of the Morley product in NIP theories (see Remark  \ref{remark:tools}). Equation $(**)$ is just Fact \ref{Fact:easy-1}. 
\end{proof}

We now given an example of a type and a point in $\mathbfcal{C}$ such that the type is not prob independent over all points in $\mathbfcal{C}$. Hence the results in this section are not just about \emph{extensions-by-definition}. 

\begin{example} Consider $\mathbfcal{M}^{[0,1]} \prec \mathbfcal{U}^{[0,1]}$ and the partition $\mathcal{A} =\{A_1,A_2,A_3,A_4\}$ where 
\begin{itemize}
    \item $A_1 = [0,1/4)$. 
    \item $A_2 = [1/4,1/2)$. 
    \item $A_3 = [1/2,3/4)$. 
    \item $A_4 = [3/4, 1]$. 
\end{itemize}
Let $a,b,c \in M$ and consider the elements
    \begin{equation*}
g_1(t)=\begin{cases}
\begin{array}{cc}
a & t\in A_1 \cup A_2,\\
b & t\in A_3 \cup A_4.
\end{array}\end{cases} 
g_2(t)=\begin{cases}
\begin{array}{cc}
a & t\in A_1 \cup A_3,\\
b & t\in A_2 \cup A_4.
\end{array}\end{cases} 
g_3(t)=\begin{cases}
\begin{array}{cc}
a & t\in A_1 \cup A_2,\\
c & t\in A_3 \cup A_4.
\end{array}\end{cases}
\end{equation*} 
We claim that the type $\tp(g_1/\mathbfcal{C})$ is prob. independent over $g_2$, but $\tp(g_1/\mathbfcal{C})$ is not prob. independent over $g_3$. We leave this as an exercise. 
\end{example}

\begin{proposition}[T NIP]\label{prop:ind} If $p \in S_{x}^{\fs}(\mathbfcal{C},\mathbfcal{M}^{\Omega})$, $q \in S_{y}^{\fs}(\mathbfcal{C},\mathbfcal{M}^{\Omega})$ and $p$ is prob. independent over $q$, then $\nu_{p \otimes q} = \nu_{p} \otimes \nu_{q}$. 
\end{proposition}

\begin{proof} Since $p$ is prob. independent over $q$, there exists a net $(h_i)_{i \in I}$ of elements in $M_0^{\Omega}$ witnessing this property. For each $i \in I$, we let $\mathcal{A}_i$ be a partition for $h_i$. Let $b \models q|_{\mathbfcal{M}^{\Omega}}$. Fix an $\mathcal{L}$-formula $\varphi(x,y,z)$ and a tuple of parameters $e \in \mathcal{U}^{z}$. Now consider the following computation: 
\begin{align*}
\nu_{p \otimes q}(\varphi(x,y,e)) &= 
(\mathbb{E}[\varphi(x,y,f_{e})])^{p \otimes q} 
\\
&= (\mathbb{E}[\varphi(x,b,f_{e})])^{p} \\
&\overset{(a)}{=} \lim_{i \in I} \sum_{A \in \mathcal{A}_i} \mathbb{P}_0(A)\mathbb{E}[\varphi(f_{h_i|_A},b,f_{e})] \\
&= \lim_{i \in I} \sum_{A \in \mathcal{A}_i} \mathbb{P}_0(A)\nu_{\tp(b/\mathbfcal{C})}(\varphi(h_i|_{A},y,e)) \\
&= \lim_{i \in I} \left( \left( \sum_{A \in \mathcal{A}_i} \mathbb{P}_0(A) \delta_{h_i|_{A}} \right) \otimes \nu_{\tp(b/\mathbfcal{C})} \right)(\varphi(x,y,e)) \\
&\overset{(b)}{=}  \left( \left( \lim_{i \in I} \sum_{A \in \mathcal{A}_i} \mathbb{P}_0(A) \delta_{h_i|_{A}} \right) \otimes \nu_{\tp(b/\mathbfcal{C})} \right)(\varphi(x,y,e)) \\
&\overset{(c)}{=}  \left( \left( \lim_{i \in I} \nu_{\tp(h_i/\mathbfcal{C})} \right) \otimes \nu_{\tp(b/\mathbfcal{C})} \right)(\varphi(x,y,e)) \\
&\overset{(d)}{=}  \left( \left(  \nu_{\lim_{i \in I}\tp(h_i/\mathbfcal{C})} \right) \otimes \nu_{\tp(b/\mathbfcal{C})} \right)(\varphi(x,y,e)) \\
&=  \left( \nu_{p}  \otimes \nu_{\tp(b/\mathbfcal{C})} \right)(\varphi(x,y,e))\\
&\overset{(e)}{=}  \left( \nu_{p}  \otimes \nu_{q} \right)(\varphi(x,y,e)).
\end{align*}
We provide the following justifications: 
\begin{enumerate}[($a$)]
    \item Hypothesis on choice of $(h_i)_{i \in I}$, i.e. witnessing prob. independence.  
    \item In NIP, the map $-\otimes \lambda (\psi(x,y))$ is continuous and thus commutes with nets (see Remark \ref{remark:tools}). 
    \item Reverse direction of Fact \ref{Fact:easy-1}.
    \item Continuity of the map $\nu_{-}$ (Proposition \ref{prop:cont-v}).
    \item Straightforward to check that $\nu_{\tp(b/\mathbfcal{C})}|_{M} = \nu_{q}|_{M}$. Similar to Proposition \ref{prop:basic2}. 
\end{enumerate}
\end{proof}

We now prove the main result of this section. 

\begin{theorem}[T NIP] Suppose $p \in S_{x}^{\fs}(\mathbfcal{C},\mathbfcal{M}^{\Omega})$ and 
$q \in S_{x}^{\fs}(\mathbfcal{C},\mathbfcal{M}^{\Omega})$. The follow are equivalent: 
\begin{enumerate}
    \item $p$ is prob. independent over $q$.
    \item $\nu_{p \otimes q} = \nu_{p} \otimes \nu_{q}$. 
    \item For any $\mathcal{L}$-formula $\varphi(x,y,z)$ and any $e \in \mathcal{U}^{z}$, for any $(h_i,\mathcal{A}_i)_{i \in I}$ such that $h_i \in M_{0}^{\Omega}$, $\mathcal{A}_i$ is a partition for $h_i$ and $\lim_{i \in I} \tp(h_i/\mathbfcal{M}^{\Omega}) = p$, 
    \begin{equation*}
        \lim_{i \in I} \sum_{A \in \mathcal{A}_i} \mathbb{P}(A) (\mathbb{E}[\varphi(f_{h_i|_{A}},y,f_{e})])^{q} = (\mathbb{E}[\varphi(x,y,f_{e})])^{p \otimes q}
    \end{equation*}
    In other words, any net from $M_{0}^{\Omega}$ which converges to $p$ witnesses the fact that $p$ is prob. independent over $q$. 
\end{enumerate}
\end{theorem}

\begin{proof} $(1) \Rightarrow (2)$ is Proposition \ref{prop:ind}. $(3) \Rightarrow (1)$ is trivial. Indeed, we are only swapping \emph{for all} and \emph{there exists}. We prove $(2) \Rightarrow (3)$. Fix a formula $\varphi(x,y,z)$, a sequence $(h_i,A_i)_{i \in I}$ with the relavant properties, and an element $e \in \mathcal{U}^{z}$. Then 
\begin{align*}
(\mathbb{E}[\varphi(x,y,f_{e})])^{p \otimes q} &= \nu_{p \otimes q}(\varphi(x,y,e))\\ &= (\nu_{p} \otimes \nu_{q})(\varphi(x,y,e)) \\&= (\nu_{\lim_{i \in I} \tp(h_i/\mathbfcal{C})} \otimes \nu_{q})(\varphi(x,y,e)) \\ &= ((\lim_{i \in I} \nu_{\tp(h_i/\mathbfcal{C})})  \otimes \nu_{q})(\varphi(x,y,e))  \\ &= \lim_{i \in I} \Big[ ( \nu_{\tp(h_i/\mathbfcal{C})}  \otimes \nu_{q})(\varphi(x,y,e)) \Big] \\ &= \lim_{i \in I} \left[ \left( \left( \sum_{A \in \mathcal{A}_i} \mathbb{P}(A) \delta_{h_i|_{A}} \right) \otimes \nu_{q} \right)(\varphi(x,y,e)) \right] \\ &= \lim_{i \in I} \left[ \sum_{A \in \mathcal{A}_i} \mathbb{P}(A) \nu_{q}(\varphi(h_i|_{A},y,f_{e})) \right]\\ &= \lim_{i \in I} \sum_{A \in \mathcal{A}_i} \mathbb{P}(A) (\mathbb{E}[\varphi(f_{h_i|_{A}},y,f_{e})])^{q}. 
\end{align*}
Many of the justifications are similar to the previous direction. We leave this as an exercise to the reader. 
\end{proof}

\begin{corollary}[T NIP] Suppose that $\mu \in \mathfrak{M}_{x}^{\fs}(\mathcal{U},M)$ and $q \in S_{x}^{\fs}(\mathbfcal{C},\mathbfcal{M}^{\Omega})$. Then $\nu_{(s_{\mu} \otimes p)} = \mu \otimes \nu_{p}$. 
\end{corollary}

\begin{proof} Follows directly from the fact that $s_{\mu}$ is prob. independent over every type, the previous proposition, and Fact \ref{prop:basic0}.
\end{proof}

\begin{remark} In practice, it is quite difficult to compute the Morley product when there are non-trivial dependencies around. Even in the case where one of the types is realized and the other is an \emph{extension-by-definition}. We consider the following computation: Let $h \in (M_{0}^{\Omega})^{y}$ with partition $\mathcal{A}$, $b \in \mathbfcal{C}^{z}$, $\mu \in \mathfrak{M}_{x}^{\inv}(\mathcal{U},M)$, where $\mu$ is Borel-definable, and $\varphi(x,y,z)$ be an $\mathcal{L}$-formula. Then we have that 
    \begin{align*}
    (\mathbb{E}[\varphi(x,y,b)])^{s_{\mu} \otimes \tp(h/\mathbfcal{C})} &= \mathbb{E}[\varphi(x,h,b)]^{s_{\mu}} \\
    &= \int_{S_{yz}(M)} F_{\mu}^{\varphi(x;yz)} d\nu_{\tp(h,b/\mathbfcal{C})} \\
    &\overset{(a)}{=} \sum_{A \in \mathcal{A}} \int_{\substack{S_{yz}(M) \\ [y = h|_{A}]}} F_{\mu}^{\varphi(x;yz)} d \left(\nu_{\tp(h,b/\mathbfcal{C})} |_{[y = h|_{A}]} \right) \\
    &\overset{(b)}{=}\sum_{A \in \mathcal{A}} \frac{\mathbb{P}_0(A)}{\mathbb{P}_0(A)}\int_{\substack{S_{yz}(M) \\ [y = h|_{A}]}} F_{\mu}^{\varphi(x;yz)} d\lambda_{A} \\
    &=\sum_{A \in \mathcal{A}} \mathbb{P}_0(A)\int_{\substack{S_{yz}(M) \\ [y = h|_{A}]}} F_{\mu}^{\varphi(x;yz)}d \left( \frac{1}{\mathbb{P}_0(A)} \lambda_{A} \right) \\
    &\overset{(c)}{=}\sum_{A \in \mathcal{A}} \mathbb{P}_0(A)\int_{\substack{S_{yz}(M) \\ [y = h|_{A}]}} \left( F_{\mu}^{\varphi(x;h|_{A},z)} \circ \pi_{z} \right) d \left( \frac{1}{\mathbb{P}_0(A)}\lambda_{A} \right) \\
    &\overset{(d)}{=}\sum_{A \in \mathcal{A}} \mathbb{P}_0(A)\int_{S_{z}(M)}  F_{\mu}^{\varphi(x;h|_{A},z)} d \left( (\pi_{z})_{*} \left( \frac{1}{\mathbb{P}_0(A)}   \lambda_{A} \right) \right).
\end{align*}

We prove the following justifications.
\begin{enumerate}[($a$)]
    \item If $\theta(x,y) = \bigvee_{A \in \mathcal{A}} y =h|_{A}$, then $\nu_{\tp(h,b/\mathbfcal{C})}(\theta(x,y)) = 1$. The sets $\{[y_i = h|_{A}]\}_{A \in \mathcal{A}}$ form a partition of $\supp(\nu_{\tp(b,h/\mathbfcal{C})})$. See Lemma \ref{lemma:support}. 
    \item For notational purposes, we let $\lambda_{A} = \left(\nu_{\tp(h,b/\mathbfcal{C})} |_{[y = h|_{A}]} \right)$. 
    \item Notice that if $r \in \supp(\frac{1}{\mathbb{P}_0(A)} \lambda_{A})$ and $(c,d) \models r(y,z)$, then $c = h|_{A}$.
    \item Standard push-forward technique. 
\end{enumerate}

The most complicated portion of the term above is the measure we are integrating with respect to, $\pi_{*}(\frac{1}{\mathbb{P}_0(A)} \lambda_{A})$. We remark that in the case where $h$ is prob. independent over $b$, it is straightforward to check that $\pi_{*}(\frac{1}{\mathbb{P}_0(A)} \lambda_{A}) = \nu_{\tp(b/\mathbfcal{C})}$, and the computation continues, resulting in $(\mathbb{E}[\varphi(x,y,b)])^{s_{(\mu \otimes \nu_{p})}}$. In the general case, it seems like this term needs to be considered on a case-by-case basis, or one needs to know more information about the fiber function, $F_{\mu}^{\varphi}$. 
\end{remark}

Finally, we observe that if $p$ is prob. independent over $q$, then the realizations of these types, as random elements, are stochastically independent. Consider the following definition. 

\begin{definition} Suppose that $p \in S_{x}(\mathbfcal{C})$ and $q  \in S_{y}(\mathbfcal{C})$. Consider $\mathbfcal{C}'$ where $\mathbfcal{C} \prec \mathbfcal{C}'$. By  \cite[Proposition 2.1.10]{andrews2015definable} (also see \cite[Theorem 3.11]{yaacov2013theories}), the model $\mathbfcal{C}'$ is isomorphic to a reduction of a nice randomization of a model $\mathcal{N}'$ (of $T$). In particular, we have that
\begin{enumerate}
    \item $\mathbfcal{C}' = (\hat{\mathcal{K}}_{\mathbfcal{C}'},\hat{\mathcal{B}}_{\mathbfcal{C}'})$.
    \item $(\Omega_{\mathbfcal{C}'},\mathcal{B}_{\mathbfcal{C}'},\mathbb{P}_{\mathbfcal{C}'})$ is a atomless probability algebra.  
    \item $\mathcal{K}_{\mathbfcal{C}'}$ is a collection of functions from $\Omega_{\mathbfcal{C}'}$ to $\mathcal{N}'$. 
    \item For every $\mathcal{L}$-formula $\varphi(x_1,...,x_n)$ and tuple $b_1,...,b_n$ from $\mathcal{K}_{\mathbfcal{C}'}$, we have that 
    \begin{equation*}
        \{t \in \Omega_{\mathbfcal{C}'}: \mathcal{N}' \models \varphi(b_1(t),...,b_n(t))\} \in \Omega_{\mathbfcal{C}'}
    \end{equation*}
    \item $\mathbfcal{C}' \models \mathbb{E}[\varphi(b_1,...,b_n)] = r \Longleftrightarrow \mathbb{P}_{\mathbfcal{C}'}(t \in \Omega_{\mathbfcal{C}'} : \mathcal{N}' \models \varphi(b_1(t),...,b_n(t))\}) = r$. 
\end{enumerate}
For any $\mathcal{L}$-formula $\varphi(x,y)$, $e \in \mathcal{U}^{y}$, and $b \in \mathbfcal{C}'$, we let 

\begin{equation*}
    [A]_{\varphi,b} = \{t \in \Omega_{\mathbfcal{C}'} : \mathcal{N}' \models  \varphi(b(t),f_e)\} \mod \mathbb{P}_{\mathbfcal{C}'}
\end{equation*}
For any $b \in \mathbfcal{C}'$ we let
\begin{equation*}
    \Sigma_{\mathbfcal{C}'}(b) = \{[A]_{\varphi,b}: \varphi \in \mathcal{L}_{x}(\mathcal{U})\}. 
\end{equation*}
Finally, we say that the types $p$ and $q$ stochastically independent if  prob. independent over $q$ if whenever $a,b \in \mathbfcal{C}'$, $a \models p$ and $b \models q$, then $\Sigma_{\mathbfcal{C}'}(a) \cap \Sigma_{\mathbfcal{C}'}(b) = \{[\emptyset],[\Omega_{\mathbfcal{C}'}]\}$. This definition is well-defined and does not depend on the choice of $a$ or $b$. 
\end{definition}

\begin{proposition} Let $p \in S_{x}^{\inv}(\mathbfcal{C},\mathbfcal{M}^{\Omega})$ and $q \in S^{\inv}_{y}(\mathbfcal{C},\mathbfcal{M}^{\Omega})$. If $p$ is prob. independent over $q$, then $p$ and $q$ are stochastically independent. 
\end{proposition}

\begin{proof} We prove by contraposition. Suppose that $p$ is not stochastically independent over $q$. Then there exists some non-trivial $[A] \in \Sigma_{\mathbfcal{C}'}(a) \cap \Sigma_{\mathbfcal{C}'}(b)$. Then there exists $\mathcal{L}(\mathcal{U})$-formula $\varphi_1(x,e_1)$ and $\varphi_2(x,e_2)$ such that 
\begin{equation*}
    [A] = [\{t \in \Omega_{\mathbfcal{C}'}: \mathcal{N}' \models \varphi_1(a(t),f_{e_1})],
\end{equation*}
and, 
\begin{equation*}
    [A] = [\{t \in \Omega_{\mathbfcal{C}'}: \mathcal{N}' \models \varphi_2(b(t),f_{e_2})].
\end{equation*}
Now notice that 
\begin{align*}
    (\nu_{p} \otimes \nu_{q})(\varphi_1(x,e_1) \wedge \varphi_2(y,e_2)) &= \nu_{p}(\varphi_1(x,e_1)) \cdot \nu_{q}(\varphi_2(y,e_2)) \\ &= \mathbb{P}_{\mathbfcal{C}'}(A) \cdot \mathbb{P}_{\mathbfcal{C}'}(A) \\ & = \mathbb{P}_{\mathbfcal{C}'}(A)^{2}.   
\end{align*}
However, 
\begin{align*}
    \nu_{p \otimes q}(\varphi_1(x,e_1) \wedge \varphi_2(y,e_2)) &= (\mathbb{E}[\varphi_1(x,f_{e_1}) \wedge \varphi_2(y,f_{e_2})])^{p\otimes q} \\ &= \mathbb{E}[\varphi_1(a,f_{e_1}) \wedge \varphi_2(b,f_{e_2})]) \\
    &=\mathbb{P}_{\mathbfcal{C}'} \Big(\{t \in \Omega_{\mathbfcal{C}'} : \mathcal{N}' \models \varphi_1(a(t),e_1) \wedge \varphi_2(b(t),e_2)\}\Big) \\
    &=\mathbb{P}_{\mathbfcal{C}'} \Big(\{t \in \Omega_{\mathbfcal{C}'} : \mathcal{N}' \models \varphi_1(a(t),e_1) \}\Big) \\
    &=\mathbb{P}_{\mathbfcal{C}'}(A). 
\end{align*}
Since $[A] \not \in \{[\emptyset],[\Omega_{\mathbfcal{C}'}]\}$, we have that $\mathbb{P}_{\mathbfcal{C}'}(A) \not \in \{0,1\}$. Hence $\mathbb{P}_{\mathbfcal{C}'}(A) \neq \mathbb{P}_{\mathbfcal{C}'}(A)^{2}$ and so $\nu_{p \otimes q} \neq \nu_{p} \otimes \nu_{q}$. 
\end{proof}

\section{Ellis semigroup of randomization}

In this last section, we make some comments concerning the Ellis semigroup of the randomization. We show that the maps $s_{-}$ is an injective imbedding from the convolution algebra of Keisler measures to semigroup of types over the randomization (under the Newelski product). We then show that if our group is definably amenable, then any minimal left ideal of the Ellis semigroup of finitely satisfiable types is a single point. Finally, we observe that if $G$ is not definably amenable, then the minimal left ideal in $(S_{x}^{\fs}(\mathbfcal{C},G^{\Omega}),*)$ is more complicated than the minimal left ideals of $(\mathfrak{M}^{\fs}_{x}(\mathcal{G},G),*)$. We recall the standard set up:

\begin{definition} Let $T$ be a first order theory expanding a group. Let $G \models T$ and $\mathcal{G}$ be a monster model such that $G \prec \mathcal{G}$. Again, $G^{\Omega} \prec \mathcal{G}^{\Omega} \prec \mathbfcal{C}$. We recall the construction of the \emph{Newelski product}: Given $p,q \in S_{x}^{\fs}(\mathbfcal{C},G^{\Omega})$, any $\mathcal{L}$-formula $\varphi(x,\bar{z})$, and $b \in \mathbfcal{C}^{z}$, we have that
\begin{equation*}
    \mathbb{E}[\varphi(x,b))]^{p*q} = \mathbb{E}[\varphi(x \cdot y,b))]^{p\otimes q}. 
\end{equation*}
We claim that if $p,q \in S_{x}^{\dagger}(\mathbfcal{C},G^{\Omega})$, then $p * q \in S_{x}^{\fs}(\mathbfcal{C},G^{\Omega})$.
\end{definition}

We now quickly recall the definition of definable convolution. 

\begin{definition}[T is NIP] Fix $\mu,\nu \in \mathfrak{M}_{x}^{\fs}(\mathcal{G},G)$. The convolution product of $\mu$ and $\nu$ is the unique measure $\mu * \nu$ in $\mathfrak{M}_{x}^{\fs}(\mathcal{G},G)$ such that for any $\mathcal{L}$-formula $\varphi(x,\bar{z})$ and parameter $e \in \mathcal{G}^{z}$, we let 
\begin{equation*}
    (\mu * \nu)(\varphi(x,e)) = (\mu_{x} \otimes \nu_{y})(\varphi(x \cdot y,e)). 
\end{equation*}
We recall that if $\mu,\nu \in \mathfrak{M}_{x}^{\fs}(\mathcal{G},G^{\Omega})$, then $\mu * \nu \in \mathfrak{M}_{x}^{\fs}(\mathbfcal{C},G^{\Omega})$. 
\end{definition}

\begin{proposition}[T NIP] The map $s_{-}:(\mathfrak{M}_{x}^{\fs}(\mathcal{U},G),*) \to (S_{x}^{\fs}(\mathbfcal{C},G^{\Omega}),*)$ is an injective embedding of semigroups.
\end{proposition}
\begin{proof}
    The map $s_{-}$ is clearly injective and continuous by Proposition \ref{prop:cont-s}. We show our map is a homomorphism. Fix an $\mathcal{L}$-formula $\varphi(x,\bar{z})$. Let $\varphi'(x;y,\bar{z}) = \varphi(x \cdot y, \bar{z})$ and fix $b \in \mathbfcal{C}^{z}$. Now notice that 
    \begin{align*}
        (\mathbb{E}[\varphi(x,b)])^{s_{\mu * \nu}} &= \int_{S_{y}(M)} F_{\mu * \nu}^{\varphi} d\nu_{\tp(b/\mathbfcal{C})} \\ &= \int_{S_{y}(M)} F_{\mu \otimes \nu}^{\varphi'} d\nu_{\tp(b/\mathbfcal{C})} \\ &= (\mathbb{E}[\varphi(x \cdot y,b)])^{s_{\mu \otimes \nu}} \\
        &\overset{(*)}{=} (\mathbb{E}[\varphi(x \cdot y,b)])^{s_{\mu} \otimes s_{\nu}} \\ &= (\mathbb{E}[\varphi(x,b)])^{s_{\mu} * s_{\nu}},
    \end{align*}
where Equation $(*)$ holds from Theorem \ref{prop:main1}. By quantifier elimination, the statement holds. 
\end{proof}

Our first proposition shows that if $G$ is NIP definably amenable, then any minimal left ideal of $S_{x}^{\fs}(\mathbfcal{C},G)$ is a single point. This observation is an extension of \cite[Proposition 4.18]{berenstein2021definable} -- if an NIP group $G$ is definably amenable, then $G^{\Omega}$ is extremely definably amenable. The general idea of the argument is similar and our proof relies on this work these modified to the appropriate setting. We also remark that this statement in known outside of the NIP case (i.e., if $G$ is definably amenabable, then $G^{\Omega}$ is extremely definably amenable \cite[Theorem 6.3]{carmona2022definably}). 

\begin{lemma}\label{lemma:support} Suppose that $h \in (M^{\Omega}_0)^{y}$ and $b \in \mathbfcal{C}^{z}$. Let $\mathcal{A}$ be a partition for $h$. Then for every $r \in \supp(\nu_{\tp(h,b/\mathbfcal{C})})$, if $(a,c) \models r(y,z)$ then there exists some $A \in \mathcal{A}$ such that $a = h|_{A}$. 
\end{lemma}

\begin{proof} By Fact \ref{Fact:easy-1}, $\nu_{\tp(h/\mathbfcal{C})} = \sum_{A \in \mathcal{A}} \mathbb{P}_0(A) \delta_{h|_{A}}$. Let $\theta(y) : = \bigvee_{A \in \mathcal{A}} (y = h|_{A})$. Notice that 
\begin{align*}
    1 &= \sum_{A \in \mathcal{A}} \mathbb{P}(\theta(y)) \\ 
    &= \nu_{\tp(h/\mathbfcal{C})}(\theta(y)) \\
    &= (\pi_{y})_{*}(\nu_{\tp(h,b/\mathbfcal{C})})(\theta(y)) \\ 
    &=\nu_{\tp(h,b/\mathbfcal{C})}(\theta(y) \wedge z=z). 
\end{align*}
Since the set $\theta(y) \wedge z=z$ has measure 1, it follows that every type in the support of $\nu_{\tp(h,b/\mathbfcal{C})}$ contains the formula $\theta(y) \wedge x =x$. This implies the intended result. 
\end{proof}

\begin{proposition}[T NIP] Suppose that $G$ is definably amenable. Fix $\mu \in \mathfrak{M}_{x}^{\fs}(\mathcal{G},G)$ such that $\mu$ is left-$G$-invariant (these exist by \cite[Theorem 3.17]{chernikov2014external}). Then for any $p \in S_{x}^{\fs}(\mathbfcal{C},G^{\Omega})$ we have that 
\begin{equation*}
    p * s_{\mu} = s_{\mu}. 
\end{equation*}
Thus $\{s_{\mu}\}$ is a minimal left ideal. As consequence $s_{\mu}|_{G^{\Omega}}$ is left-$G^{\Omega}$-invariant which implies $G^{\Omega}$ is extremely definably amenable. 
\end{proposition}
\begin{proof}
Fix an $\mathcal{L}$-formula $\varphi(x,z)$ and an element $b \in \mathbfcal{C}^{z}$.  Choose $(h_i)_{i \in I}$ such that $h_i \in M_0^{\Omega}$ and $\lim_{i \in I} \tp(h_i/\mathbfcal{C}) = p$. Let $\varphi^{+}(x;y,z) := \varphi(y \cdot x, z)$
\begin{align*}
    (\mathbb{E}[\varphi(x,b)])^{ p * s_{\mu} } &= \lim_{i \in I} (\mathbb{E}[\varphi(h_i \cdot x,b)])^{s_{\mu} }\\
    &= \lim_{i \in I} \int_{S_{yz}(M)} F_{\mu}^{\varphi^{+}} d\nu_{\tp(h_i,b/\mathbfcal{C})}\\
    &\overset{(*)}{=} \lim_{i \in I} \int_{S_{yz}(M)} \left( F_{\mu}^{\varphi} \circ \pi_{z} \right)d\nu_{\tp(h_i,b/\mathbfcal{C})}  \\
    &= \lim_{i \in I} \int_{S_{z}(M)} F_{\mu}^{\varphi}d \left( (\pi_{z})_{*}(\nu_{\tp(h_i,b/\mathbfcal{C})})\right)\\
    &=\lim_{i \in I} \int_{S_{z}(M)} F_{\mu}^{\varphi} d\nu_{\tp(b/\mathbfcal{C})}\\
    &=\int_{S_{z}(M)} F_{\mu}^{\varphi} d\nu_{\tp(b/\mathbfcal{C})}\\
    &=(\mathbb{E}[\varphi(x,b)])^{s_{\mu}}
\end{align*}
We justify Equation $(*)$. 
Notice that if $r \in \supp(\nu_{\tp(h_i,b/\mathbfcal{C})})$ and $(a,c) \models r$ then $a \in G$ by Lemma \ref{lemma:support}. Notice that since $\mu$ is left-$G$-invariant, we have 
\begin{equation*}
    F_{\mu}^{\varphi^{+}}(r) = \mu(\varphi(a \cdot x, b)) = \mu(\varphi(x, b)) = \left(F_{\mu}^{\varphi} \circ \pi_{z}\right) (r). \qedhere 
\end{equation*}
\end{proof}

The following proposition is a sanity check. 

\begin{proposition}[T NIP] Suppose that $G$ is not definably amenable. Then no minimal left ideal of $S_{x}^{\fs}(\mathbfcal{C},G^{\Omega})$ is not a singleton. In particular, $G^{\Omega}$ is not extremely amenable.
\end{proposition}

\begin{proof} Suppose there exists some $q \in S_{x}^{\fs}(\mathbfcal{C},G^{\Omega})$ such that for any $p  \in S_{x}^{\fs}(\mathbfcal{C},G^{\Omega})$, $p * q = q$. We claim that this implies $\nu_{q}|_{G}$ is $G$-left-invariant and so $G$ is definably amenable, a contradiction. 
\end{proof}

Our final proposition demonstrates that if $G$ is not definably amenable, then the minimal left ideal of $S^{\fs}(\mathbfcal{C},G^{\Omega})$ is not $\{s_{\mu}: \mu \in I\}$ where $I$ is a minimal left ideal in $(\mathfrak{M}_{x}^{\fs}(\mathcal{G},G),*)$. To be honest, I thought it was going to work out nicely and $\{s_{\mu}:\mu \in I\}$ would be a minimal ideal. As they say, \emph{you win some, you lose some.}

\begin{proposition}[T NIP] Suppose that $G$ is not definably amenable. Let $I \subseteq (\mathfrak{M}_{x}^{\fs}(\mathcal{G},G),*)$ be a minimal left ideal. We claim that $\{s_{\mu}: \mu \in I\}$ is \textbf{not} a left ideal of $(S_{x}^{\fs}(\mathbfcal{C},G^{\Omega}),*)$. 
\end{proposition}

\begin{proof} We recall from \cite{chernikov2023definable} that $I$ is a compact convex set. Fix $\mu \in I$ such that $\mu$ is an extreme point, i.e. for any $\lambda_1,\lambda_2 \in I$ and $t \in (0,1)$, 
\begin{equation*}
    \mu \neq t\lambda_1 + (1-t)\lambda_2. 
\end{equation*}
We recall that extreme points exist by the Krein-Milman theorem. Since $G$ is not definably amenable, there exists $g \in G$ such that $\mu_{g} \neq \mu$. Now notice that 
\begin{equation*}
    \mu \neq \frac{1}{2}\mu_{g^{-1}} + \frac{1}{2} \mu_{g}, 
\end{equation*}
since $\mu$ is extreme. Thus there exists an $\mathcal{L}$-formula $\theta(x,y)$ and some $b \in \mathcal{U}$ such that 
\begin{equation*}
    \mu(\theta(x,b)) \neq \left( \frac{1}{2}\mu_{g^{-1}} + \frac{1}{2}\mu_{g} \right)(\theta(x,b)). 
\end{equation*}
Now fix $A \subseteq \Omega_0$ such that $\mathbb{P}_0(A) = 1/2$ and consider the random element
\begin{equation*}
h(t)=\begin{cases}
\begin{array}{cc}
g & t\in A,\\
e & t\in A^{c}.
\end{array}\end{cases}
\end{equation*} 
Now suppose that $\{s_{\mu} : \mu \in I\}$ is a minimal left ideal. Then there exists some $\lambda \in I$ such that $s_{\lambda} = \tp(h/\mathbfcal{C}) * \mu$. But this is impossible. The simplest way to see this is via an automorphism argument. Consider an automorphism of $\mathcal{B}_0$ which sends $A \to A^{c}$. We can extend this automorphism to an automorphism of the randomization which fixes all constant random variables from $\mathcal{G}^{\Omega}$. We note that the type $s_{\lambda}$ is fixed under this automorphism while $\tp(h/\mathbfcal{C}) * s_{\mu}$ is not. Indeed, we first see that $s_{\lambda}$ is invariant. Fix a parameter $b\in \mathbfcal{C}^{y}$. Notice that $\nu_{\tp(b/\mathcal{C})}|_{M} = \nu_{\tp(\sigma(b)/\mathcal{C})}|_{M}$ because for any $\mathcal{L}$-formula $\theta(y,z)$ and parameter $d \in M^{z}$, 
\begin{align*}
    \nu_{\tp(b/\mathbfcal{C})}(\theta(x,d)) = (\mathbb{E}[\theta(b,f_d)]) &= (\mathbb{E}[\theta(\sigma(b),\sigma(f_d))]) \\ &= (\mathbb{E}[\theta(\sigma(b),f_d]) = \nu_{\tp(\sigma(b)/\mathbfcal{C})}(\theta(x,d)). 
\end{align*}
Now for any $\mathcal{L}$-formula $\psi(x,y)$, 
\begin{align*}
    (\mathbb{E}[\psi(x,b)])^{s_{\lambda}} = \int_{S_{x}(M)} F_{\lambda}^{\psi} d\nu_{\tp(b/\mathcal{C})} = \int_{S_{x}(M)} F_{\lambda}^{\psi} d\nu_{\tp(\sigma(b)/\mathcal{C})} = (\mathbb{E}[\psi(x,\sigma(b))])^{s_{\lambda}},
\end{align*}
since the integration is (formally) occurring with respect to the measures $\nu_{\tp(b/\mathcal{C})}|_{M}$ and $\nu_{\tp(\sigma(b)/\mathcal{C})}|_{M}$. Hence $s_{\lambda}$ is fixed under this automorphism. 

We now argue that $\tp(h/\mathbfcal{C}) * s_{\mu}$ is not invariant. Consider the elements
\begin{equation*}
f(t)=\begin{cases}
\begin{array}{cc}
e & t\in A,\\
g_1^{-1} & t\in A^{c}.
\end{array}\end{cases} 
\text{and }
k(t) =(f\cdot h) (t)=\begin{cases}
\begin{array}{cc}
g & t\in A,\\
g_1^{-1} & t\in A^{c}.
\end{array}\end{cases} 
\end{equation*} 
Then if $\varphi^{+}(x;yz) = \varphi(y \cdot x,z)$, 
\begin{align*}
    (\mathbb{E}[\theta(f \cdot x,f_{b}))^{\tp(h/\mathbfcal{C}) * s_{\mu}} 
    &= [\mathbb{E}[\varphi(f \cdot h \cdot x,f_{b})]^{s_{\mu}}   \\
    &= (\mathbb{E}[\varphi(k \cdot x,f_{b}))^{s_{\mu}}   \\
    &=  \int F_{\mu}^{\varphi^{+}} d\nu_{\tp(k,f_{b}/\mathbfcal{C})}\\ 
    &= \frac{1}{2} \mu(\theta(g \cdot x,b)) + \frac{1}{2} \mu(\theta(g^{-1} \cdot x, b))\\
    &\neq \mu(\theta(x,b)) \\
    &=(\mathbb{E}[\theta( x,f_{b}))^{ s_{\mu}}   \\
    &=(\mathbb{E}[\theta(f_{e} \cdot x,f_{b}))^{ s_{\mu}}   \\
    &=(\mathbb{E}[\theta(\sigma(f) \cdot h\cdot x,f_{b}))^{ s_{\mu}}   \\
    &=(\mathbb{E}[\theta(\sigma(f) \cdot x,f_{b}))^{\tp(h/\mathbfcal{C}) * s_{\mu}}  \qedhere 
\end{align*}
\end{proof}

\bibliographystyle{plain}
\bibliography{refs}

\begin{thebibliography}{10}

\bibitem{andrews2015definable}
Uri Andrews, Isaac Goldbring, and H~Jerome Keisler.
\newblock Definable closure in randomizations.
\newblock {\em Annals of Pure and Applied Logic}, 166(3):325--341, 2015.

\bibitem{andrews2019independence}
Uri Andrews, Isaac Goldbring, and H~Jerome Keisler.
\newblock Independence in randomizations.
\newblock {\em Journal of Mathematical Logic}, 19(01):1950005, 2019.

\bibitem{yaacov2008continuous}
Ita{\"\i} Ben~Yaacov.
\newblock Continuous and random vapnik-chervonenkis classes.
\newblock {\em Israel Journal of Mathematics}, 173:309--333, 2009.

\bibitem{yaacov2009transfer}
Ita{\"\i} Ben~Yaacov.
\newblock Transfer of properties between measures and random types.
\newblock {\em unpublished research note}, 2009.

\bibitem{yaacov2013theories}
Ita{\"\i} Ben~Yaacov.
\newblock On theories of random variables.
\newblock {\em Israel Journal of Mathematics}, 194:957--1012, 2013.

\bibitem{ben2009randomizations}
Ita{\"\i} Ben~Yaacov and H~Jerome~Keisler.
\newblock Randomizations of models as metric structures.
\newblock {\em Confluentes Mathematici}, 1(02):197--223, 2009.

\bibitem{ben2010continuous}
Ita{\"\i} Ben~Yaacov and Alexander Usvyatsov.
\newblock Continuous first order logic and local stability.
\newblock {\em Transactions of the American Mathematical Society},
  362(10):5213--5259, 2010.

\bibitem{berenstein2021definable}
Alexander Berenstein and Jorge~Daniel Mu{\~n}oz.
\newblock Definable connectedness of randomizations of groups.
\newblock {\em Archive for Mathematical Logic}, 60(7-8):1019--1041, 2021.

\bibitem{carmona2022definably}
Juan~Felipe Carmona and Alf Onshuus.
\newblock Definably amenable groups in continuous logic.
\newblock {\em arXiv preprint arXiv:2201.09971}, 2022.

\bibitem{chernikov2022definable}
Artem Chernikov and Kyle Gannon.
\newblock Definable convolution and idempotent keisler measures.
\newblock {\em Israel Journal of Mathematics}, 248(1):271--314, 2022.

\bibitem{chernikov2023definable}
Artem Chernikov and Kyle Gannon.
\newblock Definable convolution and idempotent keisler measures, ii.
\newblock {\em Model Theory}, 2(2):185--232, 2023.

\bibitem{chernikov2014external}
Artem Chernikov, Anand Pillay, and Pierre Simon.
\newblock External definability and groups in nip theories.
\newblock {\em Journal of the London Mathematical Society}, 90(1):213--240,
  2014.

\bibitem{NIP4}
Artem Chernikov and Sergei Starchenko.
\newblock Regularity lemma for distal structures.
\newblock {\em Journal of the European Mathematical Society},
  20(10):2437--2466, 2018.

\bibitem{CG}
Gabriel Conant and Kyle Gannon.
\newblock Remarks on generic stability in independent theories.
\newblock {\em Annals of Pure and Applied Logic}, 171(2):102736, 2020.

\bibitem{GanCon2}
Gabriel Conant and Kyle Gannon.
\newblock Associativity of the morley product of invariant measures in nip
  theories.
\newblock {\em The Journal of Symbolic Logic}, 86(3):1293--1300, 2021.

\bibitem{CGH}
Gabriel Conant, Kyle Gannon, and James Hanson.
\newblock Keisler measures in the wild.
\newblock {\em Model Theory}, 2(1):1--67, 2023.

\bibitem{CGH2}
Kyle~Gannon Gabriel~Conant and James Hanson.
\newblock Generic stability, randomizations, and nip formulas.
\newblock {\em Preprint}, 2023.

\bibitem{NIP2}
Ehud Hrushovski and Anand Pillay.
\newblock On {NIP} and invariant measures.
\newblock {\em Journal of the European Mathematical Society}, 13(4):1005--1061,
  2011.

\bibitem{NIP3}
Ehud Hrushovski, Anand Pillay, and Pierre Simon.
\newblock Generically stable and smooth measures in {NIP} theories.
\newblock {\em Transactions of the American Mathematical Society},
  365(5):2341--2366, 2013.

\bibitem{keisler1999randomizing}
H~Jerome Keisler.
\newblock Randomizing a model.
\newblock {\em Advances in Mathematics}, 143(1):124--158, 1999.

\bibitem{khanaki2022generic}
Karim Khanaki.
\newblock Generic stability and modes of convergence.
\newblock {\em arXiv preprint arXiv:2204.03910}, 2022.

\bibitem{Guide}
Pierre Simon.
\newblock {\em A guide to NIP theories}.
\newblock Cambridge University Press, 2015.

\bibitem{simon2016note}
Pierre Simon.
\newblock A note on “regularity lemma for distal structures”.
\newblock {\em Proceedings of the American Mathematical Society},
  144(8):3573--3578, 2016.

\end{thebibliography}

\end{document}